\documentclass[11pt,oneside,reqno]{amsart}

\usepackage[T1]{fontenc}
\usepackage{amssymb}
\usepackage{algpseudocode}
\usepackage{algorithm}
\usepackage{verbatim}
\usepackage{graphicx}
\usepackage{placeins}
\usepackage{listings}
\usepackage{xcolor}
\usepackage{subcaption}
\usepackage{enumitem}
\usepackage[noadjust]{cite}
\RequirePackage{dsfont} 
 \scrollmode
\allowdisplaybreaks
\usepackage{graphicx}
\usepackage{epstopdf}
\usepackage{hyperref} 
\usepackage{mathtools}
\usepackage{float}
\usepackage{lmodern}
\usepackage{tikz}
\usetikzlibrary{positioning, arrows.meta}

\usepackage{amsmath}
\usepackage{tikz}
\usepackage{enumerate}
\usepackage{xcolor}

\definecolor{codegreen}{rgb}{0,0.6,0}
\definecolor{codegray}{rgb}{0.5,0.5,0.5}
\definecolor{codepurple}{rgb}{0.58,0,0.82}
\definecolor{backcolour}{rgb}{0.95,0.95,0.92}

\lstdefinestyle{list_style}{
  backgroundcolor=\color{backcolour}, commentstyle=\color{codegreen},
  keywordstyle=\color{magenta},
  numberstyle=\tiny\color{codegray},
  stringstyle=\color{codepurple},
  basicstyle=\ttfamily\footnotesize,
  breakatwhitespace=false,         
  breaklines=true,                 
  captionpos=b,                    
  keepspaces=true,                 
  numbers=left,                    
  numbersep=5pt,                  
  showspaces=false,                
  showstringspaces=false,
  showtabs=false,                  
  tabsize=2
}

\newcommand{\xdasharrow}[2][->]{
\tikz[baseline=-\the\dimexpr\fontdimen22\textfont2\relax]{
\node[anchor=south,font=\scriptsize, inner ysep=1.5pt,outer xsep=2.2pt](x){#2};
\draw[shorten <=3.4pt,shorten >=3.4pt,dashed,#1](x.south west)--(x.south east);
}
}

\newcommand{\DEBUG}{}

\ifdefined\DEBUG%

  \def\rem#1{{\marginpar{\raggedright\scriptsize #1}}}
  \newcommand{\pmr}[1]{\rem{\color{blue}{$\bullet$ #1}}}
  \newcommand{\ppr}[1]{\rem{\color{red}{$\bullet$ #1}}}
 \else%

  \newcommand{\ppr}[1]{}
  \newcommand{\pmr}[1]{}
 \fi

\def\rho{\varrho_1}

\def\rd{\,{\mathrm d}}

\theoremstyle{plain}
\newtheorem{theorem}{Theorem}
\newtheorem{lemma}{Lemma}
\newtheorem{fact}{Fact}
\newtheorem{corollary}{Corollary}
\newtheorem{proposition}{Proposition}

\theoremstyle{definition}
\newtheorem{remark}{Remark}

\lstdefinestyle{python}{
    language=Python,
    backgroundcolor=\color{gray!5},
    basicstyle=\ttfamily\small,
    keywordstyle=\color{blue},
    stringstyle=\color{orange},
    commentstyle=\color{gray},
    showstringspaces=false,
    breaklines=true,
    frame=single,
    rulecolor=\color{gray!40},
    numbers=left,
    numberstyle=\tiny\color{gray},
    xleftmargin=15pt,
    framexleftmargin=10pt
}

\begin{document}

\title[StPINNs - DL framework for approximation of SDEs]{StPINNs - Deep learning framework for approximation of stochastic differential equations}

\author[M. Baranek]{Marcin Baranek}
\address{AGH University of Krakow,
	Faculty of Applied Mathematics,
	Al. A.~Mickiewicza 30, 30-059 Krak\'ow, Poland}
\email{mbaranek@agh.edu.pl}

\author[P. Przyby\l owicz]{Pawe{\l } Przyby\l owicz}
 \address{AGH University of Krakow,
Faculty of Applied Mathematics,
 Al. A.~Mickiewicza 30, 30-059 Krak\'ow, Poland}
 \email{pprzybyl@agh.edu.pl}
%%%%%%%%%%%%%
\begin{abstract}
In this paper, we introduce the StPINNs (stochastic physics-informed neural networks) in a systematic manner. This provides a mathematical framework for approximating the solution of stochastic differential equations (SDEs) driven by L\'evy noise using artificial neural networks.
\end{abstract}
\maketitle
\tableofcontents
%%%%%%%%%%%%%%%%%%%%%%%%%%%%%%%%%%%%%%%%%%%%%%%%%%%%%%%%%%%%%%%%%%%%
\section{Introduction}
We consider the following stochastic differential equation
\begin{equation} \label{main_equation} \left\{ \begin{array}{ll} \displaystyle{ \rd X(t) = a(t,X(t-))\rd t + \sigma \rd L(t), \ t\in [0,T]},\\ 
X(0)=x_0, \end{array} \right. 
\end{equation}
where $x_0\in\mathbb{R}^d$, $\sigma\in\mathbb{R}^{d\times m}$, the drift coefficient $a:[0,T]\times\mathbb{R}^d\to\mathbb{R}^d$ is assumed to be at least continuous, $(L(t))_{t\geq 0}$ denotes an $m$-dimensional L\'evy process, and $X(t-)=\lim\limits_{s\to t, s<t}X(s)$, see \cite{applebaum}. We investigate a non-standard approach to approximating solutions of such equations. In particular, we propose a novel deep learning-based methodology that enables artificial neural networks to learn and replicate the trajectories of the underlying SDE. Hence, the main question we pose in this paper is as follows:
\newline\newline
{\it Can artificial neural networks learn the stochastic dynamics given by the SDE \eqref{main_equation}?}
\newline\newline
We know the answer is positive for deterministic ODEs and PDEs, see, for example, \cite{PINNs_survey}, \cite{ajen}. However, learning stochastic dynamics appears to be a much more involved task. The main problem is that artificial neural networks are deterministic functions. Therefore, we must be precise about exactly what ANNs are approximating when learning solutions to SDEs. The direct approach, known from the ODE case, is not possible for SDE. In our approach, we will use the fact that the solution $X$ lives on the path space $D([0,T],\mathbb{R}^m)$ - the Skorokhod space - and can be expressed as a deterministic functional $\Psi$ of the trajectories of the driving process $L$, see \cite{PSSS_1}. This fact allows us to define a suitable loss function for artificial neural network of a feedforward type. After minimizing the loss (i.e., training), we obtain the artificial neural network that approximates the function $\Psi$. As a byproduct, we get a new approximation algorithm for the SDEs \eqref{main_equation} driven by the additive L\'evy noise.
%%Lit overview

One of the first attempts in that direction (solving SDEs via artificial neural networks) was \cite{KS_1}. The thesis \cite{KS_1} introduces a neural network-based discretization of the Wong–Zakai approximation for solving SDEs. In this approach, the Brownian motion is replaced by a smooth interpolation, transforming the SDE into a deterministic ODE system. A deep neural network is trained to approximate the ODE solutions by minimizing a loss function reflecting the SDE’s drift–diffusion dynamics. Once trained, the model can generate approximate trajectories of the underlying stochastic process, offering a data-driven alternative to classical numerical schemes.
Numerical experiments show that while the approach successfully reproduces pathwise dynamics, it currently underperforms classical schemes such as Euler–Maruyama and Milstein in terms of convergence rate and computational efficiency, indicating the need for further methodological refinement. In \cite{NS_1}, the proposed method solves stochastic partial differential equations (SPDEs) by combining the Wiener chaos expansion with neural network approximation. The Wiener chaos expansion of the solution is first truncated and deterministic or random neural networks replace the propagators (its deterministic coefficient functions). Deterministic networks are fully trained, while random networks use fixed random hidden-layer parameters, requiring only the linear output layer to be optimized. By leveraging the universal approximation property of neural networks, this framework provides theoretical convergence guarantees and practical algorithms for approximating a wide class of SPDEs. Finally, the article \cite{LPM_1} introduces SPINODE, a stochastic physics-informed neural ordinary differential equation framework designed to learn unknown hidden physics terms inside stochastic differential equations from trajectory data. It represents the hidden physics via neural networks and propagates uncertainty through the known SDE structure to obtain deterministic ODEs governing the time evolution of statistical moments (e.g., means and covariances). Using ODE solvers within the neural ODE paradigm and training via moment-matching with mini-batch gradient descent and adjoint sensitivity, the method fits the neural networks to data-estimated moments. They conclude that SPINODE is flexible and scalable for multivariate stochastic systems with multiplicative noise, but note that learning requires many repeated trajectories.

In this paper, we propose a different approach to that known from \cite{KS_1}, \cite{NS_1}, \cite{LPM_1}. Namely, instead of replacing Brownian motion $W$ (or, in our case the L\'evy process) with some approximation, we transform our semimartingale equation \eqref{main_equation} into a suitable random ODE, with a noise process embedded in the right-hand side function. This methodology allows us to cover the SDEs driven with additive L\'evy noise and is more direct than the method proposed in the papers mentioned above. Compared to the existing literature, we extend the PINN approach (see, for example, \cite{PINNS_O1}, \cite{ajen}) to define its stochastic counterpart, called {\it stochastic physics-informed neural networks (StPINNs)}. We provide a suitable theoretical framework and we present results of numerical experiments.
%
%The paper is organized as follows.
%%%%%%%%%%%%%%%%%%%%
\section{Preliminaries}
For $x\in\mathbb{R}^d$ by $\|x\|$ we mean the Euclidean norm, while for a matrix $A\in\mathbb{R}^{d\times m}$ by $\|A\|$ we mean the Frobenius norm. Let $(\Omega,\mathcal{F},\mathbb{P})$ be a complete probabilistic space, and let $(\mathcal{F}_t)_{t\geq 0}$ be a filtration on that space that satisfies the usual conditions, see \cite{Protter}. We set $\mathcal{F}_{\infty}=\sigma\Bigl(\bigcup_{t\geq 0}\mathcal{F}_t\Bigr)$. For two sub-$\sigma$-fields $\mathcal{A}_1,\mathcal{A}_2$ of $\mathcal{F}$ we define $\mathcal{A}_1\vee\mathcal{A}_2=\sigma(\mathcal{A}_1\cup\mathcal{A}_2)$. We assume that $(L(t))_{t\in [0,T]}$ is an $m$-dimensional L\'evy process on $(\Omega,\mathcal{F},\mathbb{P})$ and with respect to the filtration $(\mathcal{F}_t)_{t\geq 0}$, see Section 4, Chapter I in \cite{Protter}. By $\mathcal{D}^1\left([0,T]\times\Omega;\mathbb{R}^d\right)$ we denote the space of all $\mathbb{R}^d$-valued stochastic processes $(y(t))_{t\in [0,T]}$ that are adapted to $(\mathcal{F}_t)_{t\geq 0}$, have a.s. continuous trajectories with $(\mathcal{F}_t)_{t\geq 0}$-adapted and c\`adl\`ag  derivative processes $(y'(t))_{t\in [0,T]}$, and have the finite norm
%$\max\Bigl\{\mathbb{E}\|y(0)\|^2,\mathbb{E}\int\limits_0^T\|y(t)\|^2dt,\mathbb{E}\int\limits_0^T\|y'(t)\|^2dt\Bigr\}<+\infty$. In $\mathcal{D}^1\left([0,T]\times\Omega;\mathbb{R}^d\right)$ we consider the norm
\begin{equation}
    \|y\|_1=\Bigl(\mathbb{E}\Bigl[\sup\limits_{0\leq t\leq T}\|y(t)\|^2\Bigr]\Bigr)^{1/2}+\Bigl(\mathbb{E}\int\limits_0^T\|y'(t)\|^2dt\Bigr)^{1/2}<+\infty.
\end{equation}
(See also Lemma \ref{equiv_def_c1} in Appendix.)
%%%%%%%%%%%%%%%%%%%%%%%%%%
\\
We impose the classical assumptions for the drift coefficient $a:[0,T]\times\mathbb{R}^d\to\mathbb{R}^d$:
\begin{itemize}
    \item [(A1)] $a\in C\left([0,T]\times\mathbb{R}^d;\mathbb{R}^d\right)$
    \item [(A2)] there exists $L_a\in [0,+\infty)$ such that for all $t\in [0,T],y_1,y_2\in\mathbb{R}^d$
    \begin{equation}
        \|a(t,y_1)-a(t,y_2)\|\leq L_a\|y_1-y_2\|.
    \end{equation}
\end{itemize}
Since the L\'evy process $(L(t))_{t\in [0,T]}$ is a semimartingale (see Corollary at page 55 in \cite{Protter}), by Theorem 6, page 255 in \cite{Protter}, under the assumptions (A1), (A2) the equation \eqref{main_equation} has a unique strong solution $X$. Moreover, since $X$ has c\`adl\`ag  paths,  we get for all $t\in [0,T]$
\begin{equation}
    \int\limits_0^t a(s,X(s-))ds=\int\limits_0^t a(s,X(s))ds 
    \quad a.s.
\end{equation}
Hence, instead of \eqref{main_equation} in the rest of the paper we will be considering
\begin{equation} \label{main_equation2} \left\{ \begin{array}{ll} \displaystyle{ \rd X(t) = a(t,X(t))\rd t + \sigma \rd L(t), \ t\in [0,T]},\\ 
X(0)=x_0, \end{array} \right. 
\end{equation}
%%%%%%
\begin{remark}
We perform our analysis under the standard assumptions (A1),(A2). We stress, however, that our approach can be used for SDEs \eqref{main_equation} with, for example, discontinuous drift or with only local Lipschitz continuous drift.
\end{remark}
%%%%%%%%%%%%%%%%%%%%%%%%%%
\section{Transforming SDE to RODE}
We introduce the process
\begin{equation}
\label{transform_1}
    Y(t)=X(t)-\sigma L(t), \ t\in [0,T],
\end{equation}
and by \eqref{main_equation2} with probability one for all $t\in [0,T]$
\begin{equation}
    Y(t)=x_0+\int\limits_0^t a(s,Y(s)+\sigma L(s))ds.
\end{equation}
Hence, the stochastic process $(Y(t))_{t\in [0,T]}$ solves the following random ordinary differential equation (RODE)
\begin{equation}
\label{main_equation23}
	%\left\{ \begin{array}{ll}
	\displaystyle{
    \mathbb{P}\Bigl(\forall_{t\in [0,T)}Y'(t) = f(t,Y(t),L(t)), Y(0)=x_0\Bigr)=1},
\end{equation} 
 with
\begin{displaymath}
    f(t,y,w)=a(t,y+\sigma w), \quad (t,y,w)\in [0,T]\times\mathbb{R}^d\times\mathbb{R}^m,
\end{displaymath}
and where for $t\in[0,T)$ by $Y'(t)$ we mean the right-hand side derivative. The value of $Y'(T)$ can be defined arbitrarily. Here we take $Y'(T):=Y'(T-)$ a.s., since $Y'(T-)=f(T,Y(T),L(T-))$ is a.s. finite. The process $(Y(t))_{t\in [0,T]}$ is adapted, has   continuous trajectories,  the derivative process $(Y'(t))_{t\in [0,T]}$ has c\`adl\`ag  paths, and is also adapted. Note that the trajectories of $Y$ are, in general, of higher smoothness than the trajectories of $X$. For example, if $L$ is the Wiener process, then the trajectories of $X$ are nowhere differentiable, while the trajectories of $Y$ are even continuously differentiable on $[0,T]$ a.s.

The proof of the following fact is straightforward and therefore omitted.
\begin{fact}
\label{prop_f}
    Under the assumptions (A1), (A2) we have that $f\in C([0,T]\times\mathbb{R}^d\times\mathbb{R}^m;\mathbb{R}^d)$ and for all $t\in [0,T]$, $y_1,y_2\in\mathbb{R}^d$, $w_1,w_2\in\mathbb{R}^m$
    \begin{equation}
        \|f(t,y_1,w_1)-f(t,y_2,w_2)\|\leq C_1\Bigl(\|y_1-y_2\|+\|w_1-w_2\|\Bigr),
    \end{equation}
    and for all $t\in [0,T]$, $y\in\mathbb{R}^d$, $w\in\mathbb{R}^{m}$ 
    \begin{equation}
        \|f(t,y,w)\|\leq C_2(1+\|y\|+\|w\|),
    \end{equation}
    where
    \begin{eqnarray}
    \label{def_C1}
        &C_1=L_a\max\{1,\|\sigma\|\},\\
    \label{def_C2}
        &C_2=\max\Bigl\{\sup\limits_{0\leq t\leq T}\|f(t,0,0)\|,L_a\max\{1,\|\sigma\|\}\Bigr\}.
    \end{eqnarray}
\end{fact}
%
%\begin{lemma}
%    Under the assumptions (A1), (A2) it holds:
%    \begin{itemize}
%        \item [(i)]
%    \end{itemize}
%\end{lemma}
%
From Theorem 3.1 in \cite{PSSS_1} and Fact \ref{prop_f}, we get that there exists  a Skorokhod measurable function $\Psi:D([0,T],\mathbb{R}^m)\to D([0,T],\mathbb{R}^d)$ such that
\begin{equation}
\label{univ_rep_1}
    Y=\Psi(L),
\end{equation}
where $Y$ is the unique solution to \eqref{main_equation23}. This give us a pathwise representation of the solution $Y$ on the path space, i.e.:
\begin{equation}
    Y(\omega)=\Psi(L(\omega)),
\end{equation}
where, for a fixed $\omega\in\Omega$, $L(\omega)\in D([0,T],\mathbb{R}^m) $ is a fixed trajectory of the L\'evy process $L$. Moreover, from \eqref{transform_1} we have that
\begin{equation}
    X(\omega)=\bar\Psi(L(\omega)):=\Psi(L(\omega))+\sigma L(\omega),
\end{equation}
where $\bar\Psi:D([0,T],\mathbb{R}^m)\to D([0,T],\mathbb{R}^d)$ is also Skorokhod measurable. Therefore, 
\newline\newline
{\it the main idea is to use a deep neural network $\mathcal{N}(\mathrm{w},\cdot)$, which input will be a finite approximation of trajectories of the L\'evy process $L$, to approximate $\Psi$.} 
\newline\newline
In this paper, we focus on the feedforward architecture for $\mathcal{N}(\mathrm{w},\cdot)$. However, other approaches are also possible. (This will be a topic of our future work.)

\begin{remark}
    In the Wiener case, i.e., when $L=W$, we have that by Theorem 1 in \cite{OP_DONET} (see also \cite{CC_1}), there exists an artificial neural network $\mathcal{N}_{\varepsilon}$ that approximates the mapping $\Psi$ uniformly at the desired error level $\varepsilon>0$. In the case where $\Psi$ is only Skorokhod measurable, such results are not known at this time.
\end{remark}
%%%%%%%%%%%%%
\section{Reformulation of solving SDE as a stochastic optimization problem}
\subsection{Theoretical loss function}\noindent
\newline 
Under assumption (A1), we define the nonlinear functional
\begin{equation}
    \mathcal{\bar L}:\mathcal{D}^1\left([0,T]\times\Omega;\mathbb{R}^d\right)\to [0,+\infty],
\end{equation}
by 
\begin{equation}
    \mathcal{\bar L}(y)=\mathbb{E}\Bigl[H(y,\tau,L)\Bigr], \ y\in \mathcal{D}^1\left([0,T]\times\Omega;\mathbb{R}^d\right),
    \label{eq:loss_function} 
\end{equation}
where 
\begin{equation}
    H(y,\tau,L)=\|x_0-y(0)\|^2+T\cdot\|y'(\tau)-f(\tau,y(\tau),L(\tau))\|^2,
\end{equation}
and the random variable $\tau\sim U(0,T)$, defined on $(\Omega,\mathcal{F},\mathbb{P})$, is independent $\mathcal{F}_{\infty}$. We refer to $\mathcal{\bar L}$ as the {\it theoretical loss function}.

\begin{proposition}
\label{prop_bL_1}
Under assumption (A1) the following holds for all $y\in \mathcal{D}^1\left([0,T]\times\Omega;\mathbb{R}^d\right)$:
        \begin{itemize}
        \item [(i)] 
        \begin{displaymath}
            \mathcal{\bar L}(y)=\mathbb{E}\|x_0-y(0)\|^2+\int\limits_0^T\mathbb{E}\|y'(t)-f(t,y(t),L(t))\|^2dt,
        \end{displaymath}
        \item [(ii)]
        \begin{displaymath}
            \mathcal{\bar L}(y)=0=\inf_{u\in \mathcal{D}^1\left([0,T]\times\Omega;\mathbb{R}^d\right)}\mathcal{\bar L}(u)\Leftrightarrow
            \mathbb{P}\Bigl(\forall_{t\in [0,T)}  \ y'(t)=f(t,y(t),L(t)), \ y(0)=x_0\Bigr)=1.
            %y             \  \hbox{solves} \ \eqref{main_equation23} 
        \end{displaymath}
    \end{itemize}
\end{proposition}
\begin{proof}
For any $y\in \mathcal{D}^1\left([0,T]\times\Omega;\mathbb{R}^d\right)$, by the freezing lemma (see, for example, Lemma 4.1 in \cite{PBaldi}) and by the fact that the $\sigma$-fields $\sigma(\tau)$, $\mathcal{F}_{\infty}$ are independent, we get
    \begin{eqnarray}
        &\mathbb{E}\|y'(\tau)-f(\tau,y(\tau),L(\tau))\|^2=\mathbb{E}\Bigl[ \mathbb{E}\Bigl(\|y'(\tau)-f(\tau,y(\tau),L(\tau))\|^2\ | \ \sigma(\tau)\Bigr)\Bigr]\notag\\
        &=\mathbb{E}\Bigl[ \mathbb{E}\Bigl(\|y'(t)-f(t,y(t),L(t))\|^2\ | \ \sigma(\tau)\Bigr)\Bigl|_{t=\tau}\Bigr]\notag\\
        &=\mathbb{E}\Bigl[ \mathbb{E}\Bigl(\|y'(t)-f(t,y(t),L(t))\|^2\Bigr)\Bigl|_{t=\tau}\Bigr]=\frac{1}{T}\int\limits_0^T\mathbb{E}\Bigl(\|y'(t)-f(t,y(t),L(t))\|^2\Bigr)dt,
    \end{eqnarray}
    which ends the proof of (i).

   We now show (ii). Let us fix an arbitrary $y\in \mathcal{D}^1\left([0,T]\times\Omega;\mathbb{R}^d\right)$. \newline ('$\Leftarrow$') If the process $y$ solves \eqref{main_equation23} then directly from (i) we have
    \begin{eqnarray}
        &\mathcal{\bar L}(y)=\mathbb{E}\|x_0-y(0)\|^2+\mathbb{E}\int\limits_0^T\|y'(t)-f(t,y(t),L(t))\|^2dt=0,
    \end{eqnarray}
    and 
    \begin{equation}
        \inf\limits_{u \in \mathcal{D}^1\left([0,T]\times\Omega;\mathbb{R}^d\right)}\mathcal{\bar L}(u)\geq 0=\mathcal{\bar L}(y)\geq \inf\limits_{u \in \mathcal{D}^1\left([0,T]\times\Omega;\mathbb{R}^d\right)}\mathcal{\bar L}(u).
    \end{equation}
    ('$\Rightarrow$') If $\mathcal{\bar L}(y)=0=\inf_{u\in \mathcal{D}^1\left([0,T]\times\Omega;\mathbb{R}^d\right)}\mathcal{\bar L}(u)$ then by (i) we get that $\mathbb{E}\|y(0)-x_0\|^2=0$ and
    \begin{equation}
        \mathbb{E}\Bigl(\int\limits_0^T\|y'(t)-f(t,y(t),L(t))\|^2dt\Bigr)=0.
    \end{equation}
    This in turn imply $\mathbb{P}(y(0)=x_0)=1$  and
    \begin{equation}
        \mathbb{P}\Bigl(\int\limits_0^T\|y'(t)-f(t,y(t),L(t))\|^2dt=0\Bigr)=1.
    \end{equation}
    Therefore, by Lemma \ref{zero_lem_1}
    \begin{equation}
        \mathbb{P}\Bigl(\forall_{t\in [0,T)}  \ y'(t)=f(t,y(t),L(t))\Bigr)=1,
    \end{equation}
    since trajectories of the nonnegative process $(\|y'(t)-f(t,y(t),L(t))\|^2)_{t\in [0,T]}$ are c\`adl\`ag.
    % Hence, $y$ satisfies \eqref{main_equation23}.
\end{proof}

Below, we state sufficient conditions for the finiteness of $\mathcal{\bar L}.$
\begin{lemma}
\label{lem_bL_finit}
Assume that the drift coefficient $a:[0,T]\times\mathbb{R}^d\to\mathbb{R}^d$ satisfies (A1), (A2). 
    \begin{itemize}
        \item [(a)] If  $\int\limits_0^T\mathbb{E}\|L(t)\|^2dt<+\infty$, then for all $y\in \mathcal{D}^1\left([0,T]\times\Omega;\mathbb{R}^d\right)$ it holds
        \begin{equation}
            0\leq\mathcal{\bar L}(y)<+\infty.
        \end{equation}
        \item [(b)] If  $\int\limits_0^T\mathbb{E}\|L(t)\|^2dt=+\infty$ and, in addition,
        \begin{displaymath}
            (A3) \ \exists_{D_0\in [0,+\infty)}:\forall_{(t,z)\in [0,T]\times\mathbb{R}^d} \ \|a(t,z)\|\leq D_0
        \end{displaymath}
        then for all $y\in \mathcal{D}^1\left([0,T]\times\Omega;\mathbb{R}^d\right)$ it holds
        \begin{equation}
            0\leq\mathcal{\bar L}(y)<+\infty.
        \end{equation}
        \end{itemize}
\end{lemma}
\begin{proof}
Fix $y\in \mathcal{D}^1\left([0,T]\times\Omega;\mathbb{R}^d\right)$. We have with probability one that
    \begin{equation}
        0\leq H(y,\tau,L)\leq 2\|x_0\|^2+2\|y(0)\|^2+2T\|y'(\tau)\|^2+2T\|f(\tau,y(\tau),L(\tau))\|^2.
    \end{equation}
    By (A1), (A2) we get
    \begin{equation}
        H(y,\tau,L)\leq\max\{2,2T,6TC_2^2\}\Bigl(\|x_0\|^2+\|y(0)\|^2+\|y(\tau)\|^2+\|y'(\tau)\|^2+\|L(\tau)\|^2\Bigr).
    \end{equation}
    Hence,
    \begin{eqnarray}
        &\mathcal{\bar L}(y)\leq\max\{2,2T,6TC_2^2\}\Bigl(\|x_0\|^2+\mathbb{E}\|y(0)\|^2+\frac{1}{T}\int\limits_0^T\mathbb{E}\|y(t)\|^2dt\notag\\
        &\quad +\frac{1}{T}\int\limits_0^T\mathbb{E}\|y'(t)\|^2dt+\frac{1}{T}\int\limits_0^T\mathbb{E}\|L(t)\|^2dt\Bigr)<+\infty.
    \end{eqnarray}
  If $\int\limits_0^T\mathbb{E}\|L(t)\|^2dt=+\infty$ then by the additional assumption (A3) we obtain
  \begin{equation}
      \mathcal{\bar L}(y)\leq 2\max\{1,T,TD_0\}\Bigl(1+\|x_0\|^2+\mathbb{E}\|y(0)\|^2+\frac{1}{T}\int\limits_0^T\mathbb{E}\|y'(t)\|^2dt\Bigr)<+\infty.
  \end{equation}
\end{proof}
 %%%%%%
 Below, we prove further properties of the theoretical loss function $\mathcal{\bar L}$.
\begin{proposition}
\label{th_lf_prop_1}
    If $\mathbb{E}\int\limits_0^T\|L(t)\|^2 dt<+\infty$ and the assumptions (A1), (A2) hold then:
    \begin{itemize}
        \item [(i)] {\rm (Coercivity)} There exists $\bar C_1\in (0,+\infty)$ such that for all $y\in \mathcal{D}^1\left([0,T]\times\Omega;\mathbb{R}^d\right)$ it holds
        \begin{equation}
            \|y\|_1\leq\bar C_1(1+\mathcal{\bar L}(y))^{1/2}.
        \end{equation}
        \item [(ii)] {\rm (Lipschitz continuity)}There exists $\bar L\in (0,+\infty)$ such that for all $y_1,y_2\in \mathcal{D}^1\left([0,T]\times\Omega;\mathbb{R}^d\right)$ it holds
        \begin{equation}
            |(\mathcal{\bar L}(y_1))^{1/2}-(\mathcal{\bar L}(y_2))^{1/2}|\leq\bar L\|y_1-y_2\|_1. 
        \end{equation}
        \item [(iii)] {\rm (Consistency)} There exist $\bar C_2,\bar C_3\in (0,+\infty)$ such that for all $y\in \mathcal{D}^1\left([0,T]\times\Omega;\mathbb{R}^d\right)$ it holds
        \begin{equation}
            \bar C_2\|y-Y\|_1\leq (\mathcal{\bar L}(y))^{1/2}\leq \bar C_3\|y-Y\|_1,
        \end{equation}
        where $(Y(t))_{t\in [0,T]}$ the unique solution of \eqref{main_equation23}.
    \end{itemize}
\end{proposition}
\begin{proof}
    We start with (i). Define the residuum
    \begin{equation}
        r_y(t)=y'(t)-f(t,y(t),L(t)), \quad t\in [0,T], \ y\in \mathcal{D}^1\left([0,T]\times\Omega;\mathbb{R}^d\right),
    \end{equation}
    then
    \begin{equation}
        \mathcal{\bar L}(y)=\mathbb{E}\|y(0)-x_0\|^2+\int\limits_0^T\mathbb{E}\|r_y(t)\|^2dt.
    \end{equation}
    By Theorem 11 in \cite{HT_1} we have almost surely for all $t\in [0,T]$ that
    \begin{equation}
        y(t)=y(0)+\int\limits_0^t y'(s)ds=y(0)+\int\limits_0^t r_y(s)ds+\int\limits_0^t f(s,y(s),L(s))ds,
    \end{equation}
    since $y'(t)=r_y(t)+f(t,y(t),L(t))$. By Fact \ref{prop_f} we have for all $t\in [0,T]$ almost surely
    \begin{equation}
        \|y(t)\|\leq\max\{1,C_2\}\Biggl( \|y(0)\|+\int\limits_0^T\Bigl(1+\|r_y(s)\|+\|L(s)\|\Bigr)ds\Biggr)+C_2\int\limits_0^t\|y(s)\|ds,
    \end{equation}
    and by the Gronwall lemma
    \begin{equation}
        \|y(t)\|\leq\max\{1,C_2\}\Biggl( \|y(0)\|+\int\limits_0^T\Bigl(1+\|r_y(s)\|+\|L(s)\|\Bigr)ds\Biggr)e^{Ct},
    \end{equation}
    for all $t\in [0,T]$ almost surely. Since $\|y(0)\|\leq \|y(0)-x_0\|+\|x_0\|$, we get 
    \begin{eqnarray}
        &&\sup\limits_{0\leq t\leq T}\|y(t)\|^2\leq 4\max\{1,C_2^2\}\max\{1,3T\}\Bigl(\|x_0\|^2+T+\int\limits_0^T\|L(s)\|^2ds\notag\\
        &&\quad\quad\quad\quad\quad\quad+\|y(0)-x_0\|^2+\int\limits_0^T\|r_y(s)\|^2ds \Bigr)
    \end{eqnarray}
    almost surely. Therefore
    \begin{equation}
    \label{est_esup_coer_1}
        \Bigl(\mathbb{E}\Bigl[\sup\limits_{0\leq t\leq T}\|y(t)\|^2\Bigr]\Bigr)^{1/2}\leq D_1(1+\mathcal{\bar L}(y))^{1/2},
    \end{equation}
    where
    \begin{equation}
        D_1:=2\max\{1,C_2\}\max\{1,(3T)^{1/2}\}\max\Bigl\{\Bigl(\|x_0\|^2+T+\int\limits_0^T\mathbb{E}\|L(s)\|^2ds\Bigr)^{1/2},1\Bigr\}<+\infty.
    \end{equation}
    Since $y'(t)=r_y(t)+f(t,y(t),L(t))$, we get
    \begin{eqnarray}
        \|y'(t)\|^2\leq 2\max\{1,3C_2^2\}\Bigl(1+\sup\limits_{0\leq t\leq T}\|y(t)\|^2+\|r_y(t)\|^2+\|L(t)\|^2\Bigr),
    \end{eqnarray}
    and by \eqref{est_esup_coer_1} we have
    \begin{eqnarray}
        \label{est_esup_coer_2}
        &&\mathbb{E}\int\limits_0^T\|y'(t)\|^2dt\leq 2\max\{1,3C_2^2\}\max\{1,T\}\Biggl(1+\mathbb{E}\Bigl[\sup\limits_{0\leq t\leq T}\|y(t)\|^2\Bigr]\notag\\
        &&\quad\quad\quad\quad+\mathbb{E}\int\limits_0^T\|r_y(t)\|^2dt+\mathbb{E}\int\limits_0^T\|L(t)\|^2dt\Biggr)\notag\\
        &&\leq D_2^2(1+\mathcal{\bar L}(y)),
    \end{eqnarray}
    where
    \begin{equation}
        D_2:=2^{1/2}(D_1^2+1)^{1/2}\max\{1,3^{1/2}C_2\}\max\{1,T^{1/2}\}\Bigl(1+\mathbb{E}\int\limits_0^T\|L(t)\|^2dt\Bigr)^{1/2}<+\infty.
    \end{equation}
    From \eqref{est_esup_coer_1} and \eqref{est_esup_coer_2}, we obtain (i). 
    To show (ii) define the function
    
    \begin{equation}
\Phi:\mathcal{D}^1\left([0,T]\times\Omega;\mathbb{R}^d\right)\to L^{2,2}(\Omega;\mathbb{R}^{2d}),
    \end{equation}
    by
    \begin{equation}
        \Phi(y)=\Biggl(y(0)-x_0,T^{1/2}\Bigl(y'(\tau)-f(\tau,
        y(\tau),L(\tau))\Bigr)\Biggr)=(y(0)-x_0,T^{1/2}r_y(\tau)),
    \end{equation}
    where $L^{2,2}(\Omega;\mathbb{R}^{2d}):=L^{2}(\Omega;\mathbb{R}^{d})\times L^{2}(\Omega;\mathbb{R}^{d})$ with the norm
    \begin{equation}
        \|(X,Y)\|_{2,2}=(\|X\|^2_{L^2(\Omega;\mathbb{R}^d)}+\|Y\|^2_{L^2(\Omega;\mathbb{R}^d)})^{1/2}, \quad (X,Y)\in L^{2,2}(\Omega;\mathbb{R}^{2d}).
    \end{equation}
    Then
    \begin{equation}
        (\mathcal{\bar L}(y))^{1/2}=\|\Phi(y)\|_{2,2}.
    \end{equation}
    Hence, we have for all $y_1,y_2\in\mathcal{D}^1\left([0,T]\times\Omega;\mathbb{R}^d\right)$
    \begin{eqnarray}
        &&|(\mathcal{\bar L}(y_1))^{1/2}-(\mathcal{\bar L}(y_2))^{1/2}|=\Bigl|\|\Phi(y_1)\|_{2,2}-\|\Phi(y_2)\|_{2,2}\Bigl|\leq \|\Phi(y_1)-\Phi(y_2)\|_{2,2},
    \end{eqnarray}
    where
    \begin{equation}
        \|\Phi(y_1)-\Phi(y_2)\|^2_{2,2}=\mathbb{E}\|y_1(0)-y_2(0)\|^2+T\cdot\mathbb{E}\|r_{y_1}(\tau)-r_{y_2}(\tau)\|^2.
    \end{equation}
    We have
    \begin{equation}
        \mathbb{E}\|y_1(0)-y_2(0)\|^2\leq\mathbb{E}\Bigl[\sup\limits_{0\leq t\leq T}\|y_1(t)-y_2(t)\|^2\Bigr],
    \end{equation}
    and for all $t\in [0,T]$, by Fact \ref{prop_f}, 
    \begin{equation}
        \|r_{y_1}(t)-r_{y_2}(t)\|\leq \|y'_1(t)-y'_2(t)\|+C_1\|y_1(t)-y_2(t)\|
    \end{equation}
    with probability one. This gives
    \begin{eqnarray}
        && T\cdot\mathbb{E}\|r_{y_1}(\tau)-r_{y_2}(\tau)\|^2=T\cdot\mathbb{E}\Bigl[\mathbb{E}\Bigl(\|r_{y_1}(\tau)-r_{y_2}(\tau)\|^2 \ | \ \sigma(\tau)\Bigr)\Bigr]\notag\\
        &&=\int\limits_0^T\mathbb{E}\|r_{y_1}(t)-r_{y_2}(t)\|^2dt]\leq  2\int\limits_0^T\mathbb{E}\|y'_1(t)-y'_2(t)\|^2dt\notag\\
        &&\quad\quad\quad +2TC_1^2\mathbb{E}\Bigl[\sup\limits_{0\leq t\leq T}\|y_1(t)-y_2(t)\|^2\Bigr].
    \end{eqnarray}
    Hence, 
    \begin{eqnarray}
        &&\|\Phi(y_1)-\Phi(y_2)\|^2_{2,2}\leq (1+2TC_1^2)\cdot\mathbb{E}\Bigl[\sup\limits_{0\leq t\leq T}\|y_1(t)-y_2(t)\|^2\Bigr]\notag\\
        &&\quad\quad\quad\quad +2\int\limits_0^T\mathbb{E}\|y'_1(t)-y'_2(t)\|^2dt.
    \end{eqnarray}
    Using the inequality $\sqrt{x+y}\leq\sqrt{x}+\sqrt{y}$, $x,y\geq 0$, we arrive at
    \begin{equation}
        \|\Phi(y_1)-\Phi(y_2)\|_{2,2}\leq\bar L\|y_1-y_2\|_1,
    \end{equation}
    with $\bar L=\max\{\sqrt{2},(1+2TC_1^2)^{1/2}\}$. This finish the proof of (ii).

    For the upper bound in (iii), note that, by Corollary \ref{cor_Y_in_C1} (i), $Y\in\mathcal{D}^1\left([0,T]\times\Omega;\mathbb{R}^d\right)$ and $\mathcal{\bar L}(Y)=0$. Hence, by the Lipschitz continuity of $\mathcal{\bar L}^{1/2}$ we get for all $y\in\mathcal{D}^1\left([0,T]\times\Omega;\mathbb{R}^d\right)$
    \begin{equation}
        (\mathcal{\bar L}(y))^{1/2}=|(\mathcal{\bar L}(y))^{1/2}-(\mathcal{\bar L}(Y))^{1/2}|\leq\bar L\|y-Y\|_1.
    \end{equation}
    We now show the lower bound in (iii). Define $e_y(t):=y(t)-Y(t)$ for $y\in\mathcal{D}^1\left([0,T]\times\Omega;\mathbb{R}^d\right)$, $t\in [0,T]$. We have that $(e_y(t))_{t\in [0,T]}\in\mathcal{D}^1\left([0,T]\times\Omega;\mathbb{R}^d\right)$. Moreover, for all $t\in [0,T)$
    \begin{eqnarray}
    \label{der_ey_decomp_1}
        && e'_y(t)=y'(t)-Y'(t)=r_y(t)+\Bigl(f(t,y(t),L(t))-f(t,Y(t),L(t))\Bigr).
    \end{eqnarray}
    Then, for all $t\in [0,T]$
    \begin{eqnarray}
        && e_y(t)=e_y(0)+\int\limits_0^te_y'(s)ds\notag\\
        &&=y(0)-x_0+\int\limits_0^tr_y(s)ds+\int\limits_0^t\Bigl(f(s,y(s),L(s))-f(s,Y(s),L(s))\Bigr)ds,
    \end{eqnarray}
    and by Fact \ref{prop_f}
    \begin{eqnarray}
        &&\|e_y(t)\|\leq \|y(0)-x_0\|+\int\limits_0^t\|r_y(s)\|ds\notag\\
        &&\quad\quad\quad\quad+\int\limits_0^t\|f(s,y(s),L(s))-f(s,Y(s),L(s))\|ds\notag\\
        &&\leq \|y(0)-x_0\|+\int\limits_0^T\|r_y(s)\|ds+C_1\int\limits_0^t\|e_y(s)\|ds
    \end{eqnarray}
    almost surely. By the Gronwall lemma we get for all $t\in [0,T]$
    \begin{equation}
        \|e_y(t)\|\leq \Bigl(\|y(0)-x_0\|+\int\limits_0^T\|r_y(s)\|ds\Bigr)e^{C_1t}
    \end{equation}
    almost surely. This implies that
    \begin{equation}
        \mathbb{E}\Bigl[\sup\limits_{0\leq t\leq T}\|e_y(t)\|^2\Bigr]\leq 2\max\{1,T\}e^{2C_1 T}\cdot\mathcal{\bar L}(y).
    \end{equation}
    From \eqref{der_ey_decomp_1} we get for all $t\in [0,T)$ that
    \begin{equation}
        \|e'_y(t)\|\leq \|r_y(t)\|+C_1\|e_y(t)\|,
    \end{equation}
    and 
    \begin{eqnarray}
        &&\mathbb{E}\int\limits_0^T\|e'_y(t)\|^2dt\leq 2\int\limits_0^T\|r_y(t)\|^2 dt+2C_1T\mathbb{E}\Bigl[\sup\limits_{0\leq t\leq T}\|e_y(t)\|^2\Bigr]\notag\\
        &&\leq 2(1+2C_1T\max\{1,T\}e^{2C_1 T})\cdot\mathcal{\bar L}(y).
    \end{eqnarray}
    Hence,
    \begin{eqnarray}
        \|y-Y\|_1=\Bigl(\mathbb{E}\Bigl[\sup\limits_{0\leq t\leq T}\|e_y(t)\|^2\Bigr]\Bigr)^{1/2}+\Bigl(\mathbb{E}\int\limits_0^T\|e'_y(t)\|^2dt\Bigr)^{1/2}\leq D_3(\mathcal{\bar L}(y))^{1/2},
    \end{eqnarray}
    where
    \begin{equation}
        D_3=\sqrt{2}e^{C_1T}\Bigl(\max\{1,T^{1/2}\}+(1+2C_1T\max\{1,T\})^{1/2}\Bigr)>0.
    \end{equation}
    This ends the proof of lower bound in (iii).
\end{proof}
From the coercivity we have that $\mathcal{\bar L}(y)+\infty$ as $\|y\|_1\to+\infty$. By the consistency property, we have that the sequence of approximations $(y_n)_{n\in\mathbb{N}}$ converges to the solution $Y$ iff $\mathcal{\bar L}(y_n)\to 0$ as $n\to+\infty$.
 %%%%%%%
 \subsection{Optimization problem statement}
 By Proposition \ref{prop_bL_1}, Lemma \ref{lem_bL_finit}, and Corollary \ref{cor_Y_in_C1}, we arrive at the following result for the theoretical loss function $\mathcal{\bar L}$.
 \begin{proposition}
     \begin{itemize}
        \item [(i)] If $\mathbb{E}\int\limits_0^T\|L(t)\|^2 dt<+\infty$ and the assumptions (A1), (A2) hold then the unique solution $(Y(t))_{t\in [0,T]}$ of \eqref{main_equation23} is the global minimizer of $\mathcal{\bar L}$ in  $\mathcal{D}^1\left([0,T]\times\Omega;\mathbb{R}^d\right)$.
        \item [(ii)] If $\mathbb{E}\int\limits_0^T\|L(t)\|^2dt=+\infty$ and the assumptions (A1), (A2), (A3) hold then the unique solution $(Y(t))_{t\in [0,T]}$ of \eqref{main_equation23} is the global minimizer of $\mathcal{\bar L}$ in  $\mathcal{D}^1\left([0,T]\times\Omega;\mathbb{R}^d\right)$.
    \end{itemize}    
 \end{proposition}
By the results above, we can express the solution to \eqref{main_equation23} in terms of finding a global minimizer of $\mathcal{\bar L}$. Note, however, that the corresponding minimization problem
\begin{equation}
\label{inf_dim_rodes}
    \inf\limits_{u\in \mathcal{D}^1\left([0,T]\times\Omega;\mathbb{R}^d\right)}\mathcal{\bar L}(u),
\end{equation}
is infinite-dimensional. Hence, for the practical implementation, instead of considering \eqref{inf_dim_rodes}, we  aim at the following finite-dimensional minimization problem
\begin{equation}
\inf\limits_{\mathrm{w}\in\mathbb{R}^N}\mathcal{L}_n(\mathrm{w}),
\end{equation}
where, for a given $n\in\mathbb{N}$,
\begin{equation}
    \mathcal{L}_n(\mathrm{w})=\mathbb{E}\Bigl[H_n(\mathcal{N}(\mathrm{w},\cdot),\tau,L)\Bigr],
\end{equation}
\begin{eqnarray}
    \label{eq:finte_loss_function}
    &H_n(\mathcal{N}(\mathrm{w},\cdot),\tau,L)=\|\mathcal{N}(\mathrm{w},0,\bar L_{\Delta_n})-x_0\|^2\notag\\
    &+T\cdot\Bigl\|\frac{\partial}{\partial t}\mathcal{N}(\mathrm{w},t,\bar L_{\Delta_n})\Bigl|_{t=\tau}-f(\tau,\mathcal{N}(\mathrm{w},\tau,\bar L_{\Delta_n}),\tilde L_n(\tau))\Bigl\|^2,
\end{eqnarray}
the function $\mathcal{N}:\mathbb{R}^N\times [0,T]\times\mathbb{R}^{m\times n}\to\mathbb{R}^d$ is the feedforward artificial neural network (see, for example, \cite{ajen}) of sufficient regularity (see Theorem \ref{cea_lem_spinns}),
%
%such that for all $\mathrm{w}\in\mathbb{R}^N$, $w\in\mathbb{R}^{m\times n}$, $n\in\mathbb{N}$, the derivative wrt $t$ of 
%\begin{equation}
%    \mathcal{N}(\mathrm{w},\cdot,w)
%\end{equation}
%is c\'adl\'ag,
\begin{displaymath}
    \Delta_n=\{0=t_0<t_1<\ldots<t_n=T\},
\end{displaymath}
\begin{displaymath}
    \bar L_{\Delta_n}=\{L(t_1),L(t_2),\ldots, L(t_n)\},
\end{displaymath}
and $(\tilde L_n(t))_{t\in [0,T]}$ is: 
\begin{itemize}
    \item [-] either the L\'evy bridge approximation of $L$, i.e.: %the process $L$ conditioned on $ \bar L_{\Delta_n}$,
    $\tilde L_n(t)$ is a sample drawn from the conditional law of $L(t) \ | \ \sigma(\bar L_{\Delta_n})$
    see \cite{LB_1},
    \item [-] or the step process based on $\bar L_{\Delta_n}$, i.e.,
    \begin{equation}
        \tilde L_n(t)=\sum\limits_{j=0}^{n-1}L(t_j)\mathbf{1}_{[t_j,t_{j+1})}(t)+L(T)\mathbf{1}_{\{t_n\}}(t).
    \end{equation}
\end{itemize}
%the L\'evy bridge approximation of $L$, i.e.: the process $L$ conditioned on $ \bar L_{\Delta_n}$, see \cite{LB_1}. 
Note that, for general L\'evy processes, conditional distributions and simulation of Lévy bridges can be highly nontrivial, especially in the cases of infinite jumps activity (infinite variation). The exact L\'evy bridge simulation is possible, for example, in the case of the Wiener process, and the (compound) Poisson process. Despite of the mentioned drawbacks, the exact simulation $\tilde L_n$ from the L\'evy bridge has the following advantage
\begin{equation}
    \label{popul_loss_prop}\mathcal{L}_n(\mathrm{w})=\mathcal{\bar L}(y_{n,\mathrm{w}})
\end{equation}
where $y_{n,\mathrm{w}}(t,\omega)=\mathcal{N}(\mathrm{w},t,\bar L_{\Delta_n}(\omega))$, see the following fact.
\begin{fact}
   Let  
   \begin{equation}
        F_{(\mathrm{w},t,y)}(z)=\Bigl\|\frac{\partial}{\partial t}\mathcal{N}(\mathrm{w},t,y)-f(t,\mathcal{N}(\mathrm{w},t,y),z)\Bigl\|^2,
    \end{equation}
    for all $(\mathrm{w},t,y)\in \mathbb{R}^N\times [0,T]\times\mathbb{R}^{m\times n}$, $z\in\mathbb{R}^m$.   Then for all $w\in\mathbb{R}^N$ 
   \begin{equation}
       \mathbb{E}\Bigl[F_{(\mathrm{w},\tau,\bar L_{\Delta_n})}(L(\tau))\Bigr]=\mathbb{E}\Bigl[F_{(\mathrm{w},\tau,\bar L_{\Delta_n})}(\tilde L_n(\tau))\Bigr],
       %\mathbb{E}\Biggl[\mathbb{E}\Bigl[F_{(\mathrm{w},t,y)}(L(t)) \ | \ \sigma(\bar L_{\Delta_n})\Bigr]\Bigl|_{t=\tau,y=\bar L_{\Delta_n}}\Biggr],
   \end{equation}
   where $\tilde L_n$ is the L\'evy bridge.
   %process $L$ conditioned on $\sigma(\bar L_{\Delta_n})$ is the L\'evy bridge.
\end{fact}
\begin{proof}
    By the properties of conditional expectation and the fact that $\sigma(\tau)$ and $\sigma(L(t))\vee\sigma(\bar L_{\Delta_n})$ are independent $\sigma$-fields, we get
   \begin{eqnarray}
        &\mathbb{E}\Bigl[F_{(\mathrm{w},\tau,\bar L_{\Delta_n})}(L(\tau))\Bigr]=\mathbb{E}\Biggl[\mathbb{E}\Bigl(F_{(\mathrm{w},\tau,\bar L_{\Delta_n})}(L(\tau)) \ \Bigl| \ \sigma(\tau)\vee\sigma(\bar L_{\Delta_n})\Bigr)
        \Biggr]\notag\\
        &=\mathbb{E}\Biggl[\mathbb{E}\Bigl(F_{(\mathrm{w},t,y)}(L(t)) \ \Bigl| \ \sigma(\tau)\vee\sigma(\bar L_{\Delta_n})\Bigr)\Bigl|_{t=\tau, y=\bar L_{\Delta_n}}
        \Biggr]\notag\\
        &=\mathbb{E}\Biggl[\mathbb{E}\Bigl[F_{(\mathrm{w},t,y)}(L(t)) \ | \ \sigma(\bar L_{\Delta_n})\Bigr]\Bigl|_{t=\tau,y=\bar L_{\Delta_n}}\Biggr]\notag\\
        &=\mathbb{E}\Biggl[\mathbb{E}\Bigl[F_{(\mathrm{w},t,y)}(\tilde L_n(t)) \ | \ \sigma(\bar L_{\Delta_n})\Bigr]\Bigl|_{t=\tau,y=\bar L_{\Delta_n}}\Biggr]=\mathbb{E}\Bigl[F_{(\mathrm{w},\tau,\bar L_{\Delta_n})}(\tilde L_n(\tau))\Bigr],
    \end{eqnarray}
    since $L(t)$ and $\tilde L_n(t)$ have the same conditional law given $\sigma(\bar L_{\Delta_n})$.
\end{proof}
The adaptedness of $y_{n,\mathrm{w}}$ can be provided by enforcing causality in the neural network architecture by, for example,  plugins the inputs $L(t_j)$ up to the current time $t\in [t_j,t_{j+1})$ and filling the rest of the inputs (for $t_j> t$) with zeros. 
%Then, $y_{n,\mathrm{w}}\in\mathcal{D}^1\left([0,T]\times\Omega;\mathbb{R}^d\right)$. 

We refer to $\mathcal{L}_n$ as to {\it approximated loss function}. 

%The use of L\'evy bridges is justified by the fact that  for all $A\in\mathcal{B}(R^m)$
%\begin{equation}
%    \mathbb{P}\Bigl(L(\tau)\in A \ | \ \sigma( \bar L_{\Delta_n})\vee\sigma(\tau)\Bigr)=\mathbb{P}\Bigl(L(t)\in A \ | \ \sigma( \bar L_{\Delta_n})\Bigr)\Bigl|_{t=\tau}
%\end{equation}
%since $\tau$, $L$ are independent. Moreover, we have
%%%%%%%%%%%%%%
\begin{remark}
    In general, the theoretical loss function $\mathcal{\bar L}$ is not convex, even for $\sigma=0$. Consider the following counterexample. Set $\sigma=0$, $a(t,x)=\sqrt{x^2+\varepsilon}-1$, $\varepsilon=0.01$, $d=1$, $T=1$, $x_0=0$. Then $a=a(t,x)$ is Lipschitz continuous wrt $x$. Take $y_c(t,\omega)=c\in\mathbb{R}$ for all $t\in [0,1]$, $\omega\in\Omega$. Then $y_c\in\mathcal{D}^1([0,T]\times\Omega;\mathbb{R}^d)$ for all $c\in\mathbb{R}$, $\mathcal{\bar L}(y_c)=c^2+(\sqrt{c^2+\varepsilon}-1)^2$, and for $c_1=-0.1$, $c_2=-c_1=0.1$ we get that
    \begin{equation}
        \mathcal{\bar L}\Bigl(\frac{y_{c_1}+y_{c_2}}{2}\Bigr)-\frac{\mathcal{\bar L}(y_{c_1})+\mathcal{\bar L}(y_{c_2})}{2}=\frac{20\sqrt{2}-22}{100}>0.
    \end{equation}
\end{remark}
%%%%%%%%%%%
\subsection{SGD algorithm}
%%%%%%%%%%
For fixed $n\in\mathbb{N}$, by applying the Robbins-Monro stochastic approximation algorithm \cite{robsieg}, we get
\begin{eqnarray}
\label{SGD_H_n}
    &\mathrm{w}^0\in\mathbb{R}^N,\notag\\
    &\mathrm{w}^{k+1}=\mathrm{w}^k-\eta_k\cdot\nabla_{\mathrm{w}} H_n(\mathcal{N}(\mathrm{w},\cdot),\tau_k,L_k)\Bigl|_{\mathrm{w}=\mathrm{w}^k}
\end{eqnarray}
for $k=0,1,2,\ldots$, where $(\tau_k)_{k\geq 0}$ is an iid sequence from $U(0,T)$, $(L_k)_{k\geq 0}$ is a sequence of independent L\'evy processes (that is also independent of $(\tau_k)_{k\geq 0}$), and we impose the following assumption on the sequence of learning rates $\{\eta_k\}_{k\geq 0}$
\begin{displaymath}
    \sum\limits_{k\geq 0}\eta_k=+\infty, \ \sum\limits_{k\geq 0}\eta^2_k<+\infty.
\end{displaymath}
After training the neural network and obtaining $\textrm{w}^*$, we get the neural network such that 
\begin{equation}
    \mathcal{N}(\mathrm{w}^*,t,\bar L_{\Delta_n}(\omega))\approx (\Psi(L(\omega)))(t),
\end{equation}
and we define the approximation to $X$ as
\begin{equation}
    \bar X_n(t,\omega)=\mathcal{N}(\mathrm{w}^*,t,\bar L_{\Delta_n}(\omega))+\sigma \tilde L_n(t,\omega), \quad t\in [0,T], \omega\in\Omega.
\end{equation}
%and for $t\in [0,T]\setminus\Delta_n$ we take $L(t,\omega)$ as the piecewise linear interpolation of $\bar L_{\Delta_n}(\omega)$ or by sampling the L\'evy bridge.

The resulting scheme, together with the above derivations, introduces the so-called {\it Stochastic Physics-Informed Neural Networks} (StPINNs). We have the following C\'ea-type quasi-optimality result.
\begin{theorem}
\label{cea_lem_spinns}
        Let the assumptions (A1), (A2) hold, $y_{n,\mathrm{w}}\in\mathcal{D}^1\left([0,T]\times\Omega;\mathbb{R}^d\right)$ for all $\mathrm{w}\in\mathbb{R}^N$,  $n\in\mathbb{N}$. Moreover, assume that the exact L\'evy bridge simulation is possible. Then there exists $\gamma\in (0,+\infty)$ such that for all $\mathrm{w}^*\in\mathbb{R}^N$, $n\in\mathbb{N}$
    \begin{equation}
        \|y_{n,\mathrm{w}^*}-Y\|_1\leq \gamma\Bigl(\inf\limits_{\mathrm{w}\in\mathbb{R}^N}\|y_{n,\mathrm{w}}-Y\|_1+(\mathcal{L}_n(\mathrm{w}^*)-\inf\limits_{\mathrm{w}\in\mathbb{R}^N}\mathcal{L}_n(\mathrm{w}))^{1/2}\Bigr)
    \end{equation}
    and if $\mathrm{w}^*\in argmin_{\mathrm{w}\in\mathbb{R}^N} \ \mathcal{L}_n(\mathrm{w})$ then
    \begin{equation}
        \|y_{n,\mathrm{w}^*}-Y\|_1\leq \gamma\inf\limits_{\mathrm{w}\in\mathbb{R}^N}\|y_{n,\mathrm{w}}-Y\|_1.
    \end{equation}
\end{theorem}
\begin{proof}
    By \eqref{popul_loss_prop} and  Proposition \ref{th_lf_prop_1} (iii) we have that for all $\mathrm{w}^*\in\mathbb{R}^N$, $n\in\mathbb{N}$, 
    \begin{eqnarray}
        &&\|y_{n,\mathrm{w}^*}-Y\|_1\leq\frac{1}{\bar C_2}(\mathcal{\bar L}(y_{n,\mathrm{w}^*}))^{1/2}=\frac{1}{\bar C_2}(\mathcal{L}_n(\mathrm{w}^*))^{1/2}\notag\\
        &&\leq\frac{1}{\bar C_2}\Bigl((\inf\limits_{\mathrm{w}\in\mathbb{R}^N}\mathcal{L}_n(\mathrm{w}))^{1/2}+(\mathcal{L}_n(\mathrm{w}^*)-\inf\limits_{\mathrm{w}\in\mathbb{R}^N}\mathcal{L}_n(\mathrm{w}))^{1/2}\Bigr),
    \end{eqnarray}
    and
    \begin{equation}
        (\inf\limits_{\mathrm{w}\in\mathbb{R}^N}\mathcal{L}_n(\mathrm{w}))^{1/2}=(\inf\limits_{\mathrm{w}\in\mathbb{R}^N}\mathcal{\bar L}(y_{n,\mathrm{w}}))^{1/2}\leq\bar C_3\inf\limits_{\mathrm{w}\in\mathbb{R}^N}\|y_{n,\mathrm{w}}-Y\|_1.
    \end{equation}
    Combining the inequalities above, we get the thesis.
\end{proof}
Note that when $\sigma=0$, then we recover the classical PINNs, see \cite{ajen}. In this special case the neural network that we train does not depend on $\bar L_{\Delta_n}$, so $\mathcal{N}:\mathbb{R}^N\times [0,T]\to\mathbb{R}^d$. Moreover, $X=Y\in C^1([0,T];\mathbb{R}^d)$, $\|X\|_1=\sup\limits_{0\leq t\leq T}\|X(t)\|+\Bigl(\int\limits_0^T\|X'(t)\|^2dt\Bigr)^{1/2}$ and the approximate loss $\mathcal{L}(\mathrm{w})=\mathcal{\bar L}(\mathcal{N}(\mathrm{w},\cdot))$ does not depend on $n$. Hence, we have the following corollary.
\begin{corollary}
        Let $\sigma=0$, let the assumptions (A1), (A2) hold and assume that $\mathcal{N}(\mathrm{w},\cdot)\in C^1([0,T];\mathbb{R}^d)$ for all $\mathrm{w}\in\mathbb{R}^N$. Then there exists $\gamma\in (0,+\infty)$ such that for all $\mathrm{w}^*\in\mathbb{R}^N$
    \begin{equation}
        \|\mathcal{N}(\mathrm{w}^*,\cdot)-Y\|_1\leq \gamma\Bigl(\inf\limits_{\mathrm{w}\in\mathbb{R}^N}\|\mathcal{N}(\mathrm{w},\cdot)-Y\|_1+(\mathcal{L}(\mathrm{w}^*)-\inf\limits_{\mathrm{w}\in\mathbb{R}^N}\mathcal{L}(\mathrm{w}))^{1/2}\Bigr)
    \end{equation}
    and if $\mathrm{w}^*\in argmin_{\mathrm{w}\in\mathbb{R}^N} \ \mathcal{L}(\mathrm{w})$ then
    \begin{equation}
        \|\mathcal{N}(\mathrm{w}^*,\cdot)-Y\|_1\leq \gamma\inf\limits_{\mathrm{w}\in\mathbb{R}^N}\|\mathcal{N}(\mathrm{w},\cdot)-Y\|_1.
    \end{equation}
\end{corollary}
%\begin{proof}
 %   By Proposition \ref{th_lf_prop_1} (iii) we have that for all $\mathrm{w}^*\in\mathbb{R}^N$
 %   \begin{equation}
 %       \|\mathcal{N}(\mathrm{w}^*,\cdot)-Y\|_1\leq\frac{1}{\bar C_2}(\mathcal{L}(\mathrm{w}^*))^{1/2}\leq\frac{1}{\bar C_2}\Bigl((\inf\limits_{\mathrm{w}\in\mathbb{R}^N}\mathcal{L}(\mathrm{w}))^{1/2}+(\mathcal{L}(\mathrm{w}^*)-\inf\limits_{\mathrm{w}\in\mathbb{R}^N}\mathcal{L}(\mathrm{w}))^{1/2}\Bigr),
  %  \end{equation}
  %  and
  %  \begin{equation}
   %     (\inf\limits_{\mathrm{w}\in\mathbb{R}^N}\mathcal{L}(\mathrm{w}))^{1/2}\leq\bar C_3\inf\limits_{\mathrm{w}\in\mathbb{R}^N}\|\mathcal{N}(\mathrm{w},\cdot)-Y\|_1.
   % \end{equation}
    %Combining the inequalities above, we get the thesis.
%\end{proof}
The main differences between the derivation of StPINNs and PINNs are the use of the universal representation theorem that gives us the existence of the function $\Psi$ in \eqref{univ_rep_1}, and the construction of the loss function for the artificial neural network $\mathcal{N}(\mathrm{w},\cdot,\cdot)$ that, after minimization, allows us to approximate the function $\Psi$. It was not the case in the classical PINNs.  

\begin{remark}
    The results presented in this paper can be applied directly, under the assumptions (A1), (A2), to the SDEs \eqref{main_equation} driven by the L\'evy process $L$ of a finite L\'evy measure. In this case, $L$ can be, for example, a Wiener process, a compound Poisson process, or their linear combination. If the drift coefficient $a$ satisfies (A1), (A2), (A3), then we can consider the case when, for example, $L$ is the Cauchy process.
\end{remark}
%\begin{remark}
 %   Since the L\'evy process $L$ has independent increments, we can use the well-known incremental method to simulate values of $L_k(t_j)$, $j=1,2,\ldots,n$, $k\geq 0$.
%\end{remark}
\begin{remark}
    There is another option of approximating the loss $\mathcal{\bar L}$. Namely, we can take
    \begin{eqnarray}
        &\tilde H_n(\mathcal{N}(\mathrm{w},\cdot),L)=\|\mathcal{N}(\mathrm{w},0,\bar L_{\Delta_n})-x_0\|^2\notag\\
    &+\frac{T}{n}\sum\limits_{j=0}^{n-1}\Bigl\|\frac{\partial}{\partial t}\mathcal{N}(\mathrm{w},t,\bar L_{\Delta_n})\Bigl|_{t=t_j}-f(t_j,\mathcal{N}(\mathrm{w},t_j,\bar L_{\Delta_n}),L(t_j))\Bigl\|^2.
    \end{eqnarray}
Such a loss function can be used in a more general case when $L$ is a c\'adl\'ag semimartingale, since for such a $\tilde H_n$, we do not need to know the conditional law $L(t)$ given $\bar L_{\Delta_n}$.     %However, under such an approach, we cannot use a compact form SGD \eqref{SGD_H_n}. 
\end{remark}
%%%%%%%%%%%%%
%\subsection{Some insight into the multiplicative noise case - Lamperti transformation}
%..............
\subsection{Some insight into the multiplicative noise case - Doss-Sussman
transformation}
Our diffusion removal transformation is only suitable for the SDEs with the additive noise as in \eqref{main_equation} and cannot be directly applied to the SDEs with the multiplicative noise. In the scalar case and $L=W$ being a scalar Wiener process let us consider the SDE
\begin{equation} \label{scalar_m_equation4} \left\{ \begin{array}{ll} \displaystyle{ \rd X(t) = a(t,X(t))\rd t + \sigma X(t)dW(t), \ t\in [0,T]},\\ 
X(0)=x_0>0, \end{array} \right. 
\end{equation}
To transform the above equations into the RODE of the form \eqref{main_equation23} we proceed as follows. First, consider the SDE
\begin{equation} \label{scalar_m_equation5} \left\{ \begin{array}{ll} \displaystyle{ \rd Z(t) = (e^{-Z(t)}a(t,e^{Z(t)})-\sigma^2/2)\rd t + \sigma dW(t), \ t\in [0,T]},\\ 
Z(0)=\ln(x_0), \end{array} \right. 
\end{equation}
and then use the transformation \eqref{transform_1}, i.e., $Y(t)=Z(t)-\sigma W(t)$, to transform the SDE \eqref{scalar_m_equation5} to the RODE
\begin{equation}
\label{main_equation231}
	%\left\{ \begin{array}{ll}
	\displaystyle{
    \mathbb{P}\Bigl(\forall_{t\in [0,T)}Y'(t) = f(t,Y(t),W(t)), Y(0)=\ln(x_0)\Bigr)=1},
\end{equation} 
 where
\begin{displaymath}
    f(t,y,w)=e^{-(y+\sigma w)}a(t,e^{y+\sigma w})-\sigma^2/2, \quad (t,y,w)\in [0,T]\times\mathbb{R}\times\mathbb{R}.
\end{displaymath}
Then, we can apply our StPINN method to approximate $Y$. Since $X(t)=e^{Z(t)}$, what can be verified by the It\^o formula, finally, we get the approximation of $X$. 

Note that both transformations are the particular case of the well-known Doss-Sussman transformation, see \cite{Doss_1}, \cite{Suss_1}. In our future work we investigate the possibility of applying the D-S transformation to the multidimensional SDEs with multiplicative noise to extend the StPINNs approach derived in this paper.
%%%%%%%%%%%%%
\section{Numerical experiments}
%%%%%%%%
In this section, we focus on numerical experiments. All the code is available at 
\newline
\url{https://github.com/MarcinBaranek/Research/tree/master/SDE-ANN-solver}
\newline
We use a discretization of the time interval $[0,T]$ as a uniform mesh with $n$ elements.
For numerical purpose we take $L$ as the Wiener process.
As the exact solution $X$ of \eqref{main_equation} we treat the solution from the Euler algorithm calculated with precision $float64$ and the number of discretization points equal to $2^{17}$. By $\bar{X}_n$ we mean the solution calculated by the neural network with precision $float32$ and information about $L$ at $n$ points.
To see how perform the network we define errors
\[
    err_{2,n}([0,T]) 
    = \sqrt{\frac{1}{M}\sum_{i=1}^M 
        \frac{T}{n}\sum_{j=1}^n\vert \bar{X}(t_j, L_i) - X(t_j,L_i)\vert^2}\approx\Bigl(\mathbb{E}\int\limits_0^T|\bar X_n(s)-X(s)|^2ds\Bigr)^{1/2},
\]
and the error at the end-point $T$ is approximated by
\[
    err_{2,n}(T) 
    = \sqrt{\frac{1}{M}\sum_{i=1}^M 
        \vert \bar{X}(T, L_i) - X(T,L_i)\vert^2}\approx (\mathbb{E}|\bar X_n(T)-X(T)|^2)^{1/2}.
\]
Those errors are calculated with precision $float64$. We take $M=10000$, unless specified otherwise.

We consider three exemplary equations, where for the first two $L$ is a scalar Wiener process, and for the last example $L$ is a four dimensional Wiener process.
\begin{itemize}[leftmargin=3cm]
    \item[Example 1.]  \begin{equation}
    \label{eq:example1}
    \left\{ \begin{array}{ll} \displaystyle{ \rd X(t) = 5(0.4-X(t-))\rd t + 0.61 \rd L(t), \ t\in [0,1]},\\ 
X(0)=-0.3. \end{array} \right. 
\end{equation}
    \item[Example 2.]  \begin{equation} 
    \label{eq:example2}
    \left\{ \begin{array}{ll} \displaystyle{ \rd X(t) = 5(0.4-\sin(X(t-)))\rd t + 0.61 \rd L(t), \ t\in [0,1]},\\ 
X(0)=-0.3. \end{array} \right. 
\end{equation}
\item[Example 3.]
    \begin{equation}
        \left\{
        \begin{array}{ll} \displaystyle{
            dX_t = 
        5\left(\begin{bmatrix}0.4\\0.7\\-1.2\end{bmatrix} - X_t\right) \, dt 
        + \begin{bmatrix}
            0.61&0.3&0&0.1\\
            -0.3&-0.72&0.45&0\\
            0&0&0.34&-0.81
          \end{bmatrix} \, dL_t,\quad t\in[0,1]},\\
          X(0)= \begin{bmatrix}
              -0.3\\-0.3\\-0.3
          \end{bmatrix}.
          \end{array}
        \right.
    \end{equation}
\end{itemize}
%%%%%%%%%%
\subsection{Implementation details}

We employ a fully connected deep neural network with three hidden layers, using $\tanh$ as the activation function. A sample implementation of this network is shown in Listing \ref{lst:network}. We assume that the network input has shape $(\mathrm{Batch\ size}, 1+n)$, where $n$ is the number of points in a uniform mesh over the interval $[0, T]$. The output of the network has shape $(\mathrm{Batch\ size}, 1)$.

\begin{lstlisting}[style=python, caption={Sample of code implementing network}, label={lst:network}]
import keras

class Network(keras.Model):
    def __init__(
            self, n_points: int, initial_value: float,
            derivative_initial_value: float
    ):
        super().__init__()
        hidden_layers = [
            keras.layers.Dense(n_points, activation="hard_silu"),
            keras.layers.Dense(512, activation="hard_silu"),
            keras.layers.Dense(256, activation="hard_silu"),
            keras.layers.Dense(128, activation="hard_silu"),
        ]
        self.initial_value = initial_value
        self.derivative_initial_value = derivative_initial_value
        self.hidden_layers = hidden_layers
        self.out = keras.layers.Dense(1)

    def call(self, x):
        t = x[..., :1]
        h = x
        for layer in self.hidden_layers:
            h = layer(h)
        f_theta = self.out(h)
        return self.initial_value\
            + self.derivative_initial_value * t\
            + f_theta * t ** 2 / 2.
\end{lstlisting}

The neural network $\mathcal{N}$ is trained to minimize the loss function defined in \eqref{eq:finte_loss_function}. To accelerate the training process and enforce the initial condition, we constrain the network to satisfy
\[\mathcal{N}(\mathrm{w},0,\bar{L}_{\Delta_n})=x_0.\]
This constraint is imposed through the transformation
\[\bar{\mathcal{N}}(\mathrm{w},t,\bar{L}_{\Delta_n})=x_0 + t\mathcal{N}(\mathrm{w},t,\bar{L}_{\Delta_n}).\]
which guarantees that the initial condition is satisfied identically. As a consequence, the first term in \eqref{eq:finte_loss_function} is identically zero for all network parameters.
Furthermore, prior to training, the expected value of the time derivative of the solution at $t=0$ is known and equals $f(0,x_0,0)$. By combining this information with the imposed initial condition, the neural network architecture can be expressed as
\begin{equation*}
    \bar{\mathcal{N}}(\mathrm{w},t,\bar{L}_{\Delta_n})=x_0 +tf(0,x_0,0) +\frac{t^2}{2}\mathcal{N}(\mathrm{w},t,\bar{L}_{\Delta_n}).
\end{equation*}
This behavior is implemented in lines 31–33 of Listing \ref{lst:network}.

\begin{figure}[H]
        \centering
        \begin{tikzpicture}[
            >=Stealth,
            node distance=0.5cm,
            every node/.style={font=\small},
            box/.style={
                rectangle,
                draw,
                rounded corners,
                minimum width=1.5cm,
                minimum height=1.0cm,
                align=center
            },
            neural/.style={box, fill=blue!10},
            physics/.style={box, fill=orange!15}
        ]
        
        % --- Top row (left to right) ---
        \node[box] (input) {Input $(t,\bar{L}_{\Delta_n})$};
        \node[neural, right=of input] (d1) {Dense (1024) \\ Hard-SiLU};
        \node[neural, right=of d1] (d2) {Dense (512) \\ Hard-SiLU};
        
        % --- Bottom row (right to left) ---
        \node[neural, below=0.5cm of d2] (d3) {Dense (256) \\ Hard-SiLU};
        \node[neural, left=of d3] (d4) {Dense (128) \\ Hard-SiLU};
        \node[box, left=of d4] (out) {Dense (d) \\ $\bar{\mathcal{N}}(\mathrm{w},t,\bar{L}_{\Delta_n})$};
        
        % --- Final physics block ---
        \node[physics, below=0.6cm of d4, minimum width=3cm] (formula)
        {
        $\displaystyle
        \mathcal{N}(\mathrm{w},t,\bar{L}_{\Delta_n})=x_0 +tf(0,x_0,0) +\frac{t^2}{2}\bar{\mathcal{N}}(\mathrm{w},t,\bar{L}_{\Delta_n})
        $
        };
        
        % --- Arrows (zig-zag flow) ---
        \draw[->] (input) -- (d1);
        \draw[->] (d1) -- (d2);
        \draw[->] (d2) -- (d3);      % down
        \draw[->] (d3) -- (d4);
        \draw[->] (d4) -- (out);
        \draw[->] (out) -- (formula);
        
        \end{tikzpicture}
        \caption{Example network architecture.}
        \label{fig:architecture}
\end{figure}
In Figure \ref{fig:architecture} we present an example of network architecture. The number of neurons in each dense layer is our guess, and we believe the similar result are obtainable with different architecture.

\subsection{training details.}
We train the network using randomly sampled trajectories of $L$ at $n$ points, i.e., $\bar L_{\Delta_n}$, together with independently generated time values $\tau$. The loss function is defined in Equation \eqref{eq:loss_function}. In our setup, one epoch corresponds to training the network on a batch of 64 sampled trajectories of $L$.

\begin{figure}[htbp]
\centering

\begin{subfigure}{0.45\textwidth}
    \centering
    \includegraphics[width=\linewidth]{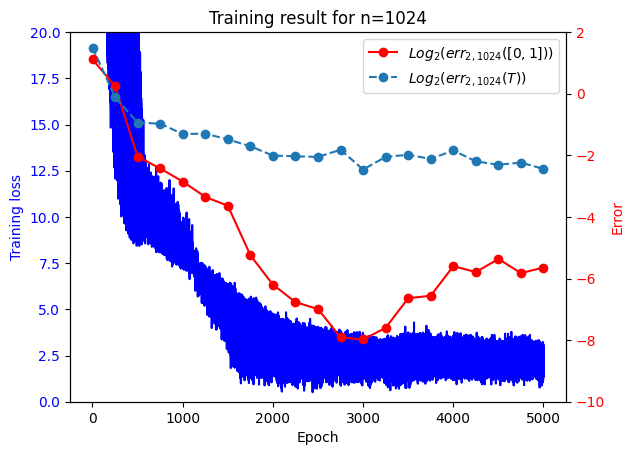}
\end{subfigure}
\hfill
\begin{subfigure}{0.45\textwidth}
    \centering
    \includegraphics[width=\linewidth]{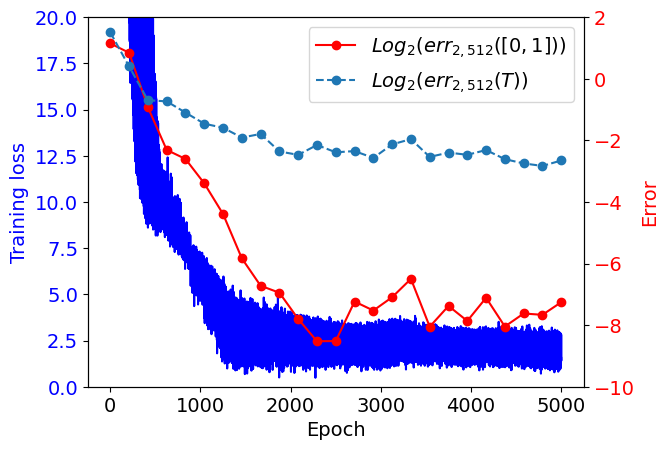}
\end{subfigure}
\vspace{0.1cm}
\begin{subfigure}{0.45\textwidth}
    \centering
    \includegraphics[width=\linewidth]{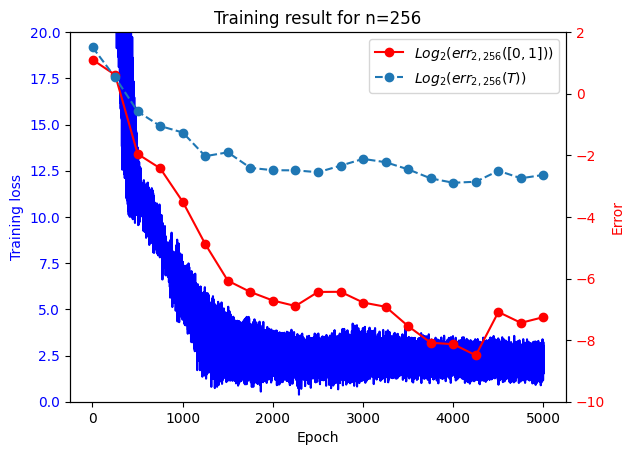}
\end{subfigure}
\hfill
\begin{subfigure}{0.45\textwidth}
    \centering
    \includegraphics[width=\linewidth]{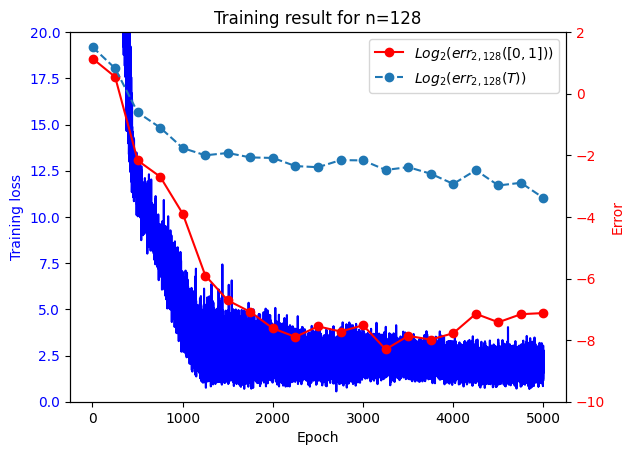}
\end{subfigure}

\caption{Training loss and approximated errors for example 1. with different value of $n$.}
\label{fig:train_loss}
\end{figure}

% \begin{figure}[H]
%     \centering
%     \includegraphics[width=0.8\linewidth,scale=0.75]{plots/train_ex1.png}
%     \caption{Training loss and approximated errors for example 1.}
%     \label{fig:train_loss}
% \end{figure}

In Figure \ref{fig:train_loss}, we observe the evolution of the training loss and the approximated errors. The approximation errors are plotted on a logarithmic scale. We hypothesize that the neural network is unlikely to overfit, provided that we use randomly sampled trajectories of $L$, therefor the network is unable to memorize the training set.
The hypothesis is supported empirically by Figure \ref{fig:train_loss}. The training loss stabilizes over time, and the approximation error exhibits a similar behavior. Consequently, no increase in the approximation error is observed during training.

The network $\mathcal{N}$ appears to be 
    \emph{overly influenced} by the trajectories $\bar{L}_{\Delta_n}$ 
    and struggles to solve the ODE (case when $(\bar{L}_{\Delta_n} \equiv 0)$).
    The input contains a lot of numbers representing $\bar{L}_{\Delta_n}$ and only one representing the time, therefore the network may doesn't learn how important the time is. The example solution is in Figure \ref{fig:ode_dummy}.
To make the network paying a more attention to the time, we train it with randomly generated trajectories $\bar{L}_{\Delta_n}$, perform weights optimization, train with trivial trajectories $\bar{L}_{\Delta_n}\equiv0$, perform weights optimization.
The result in the ODE case are presenting in Figure \ref{fig:ode_smart}.
\begin{figure}
    \centering    
    \begin{subfigure}[b]{0.48\textwidth}
        \centering
        \includegraphics[width=1\linewidth]{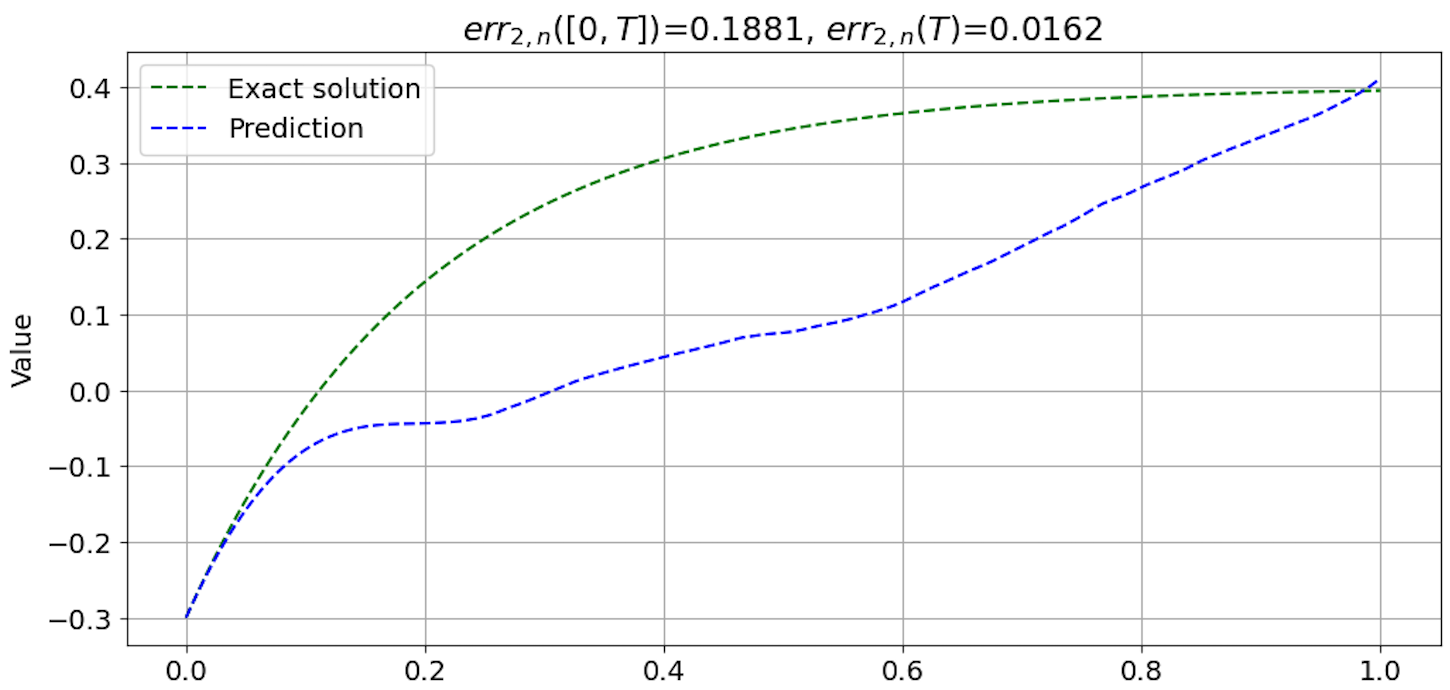}
        \caption{Network performance in solving ODE after standard training procedure.}
        \label{fig:ode_dummy}
    \end{subfigure}
    \hfill
    \begin{subfigure}[b]{0.48\textwidth}
        \centering
        \includegraphics[width=1\linewidth]{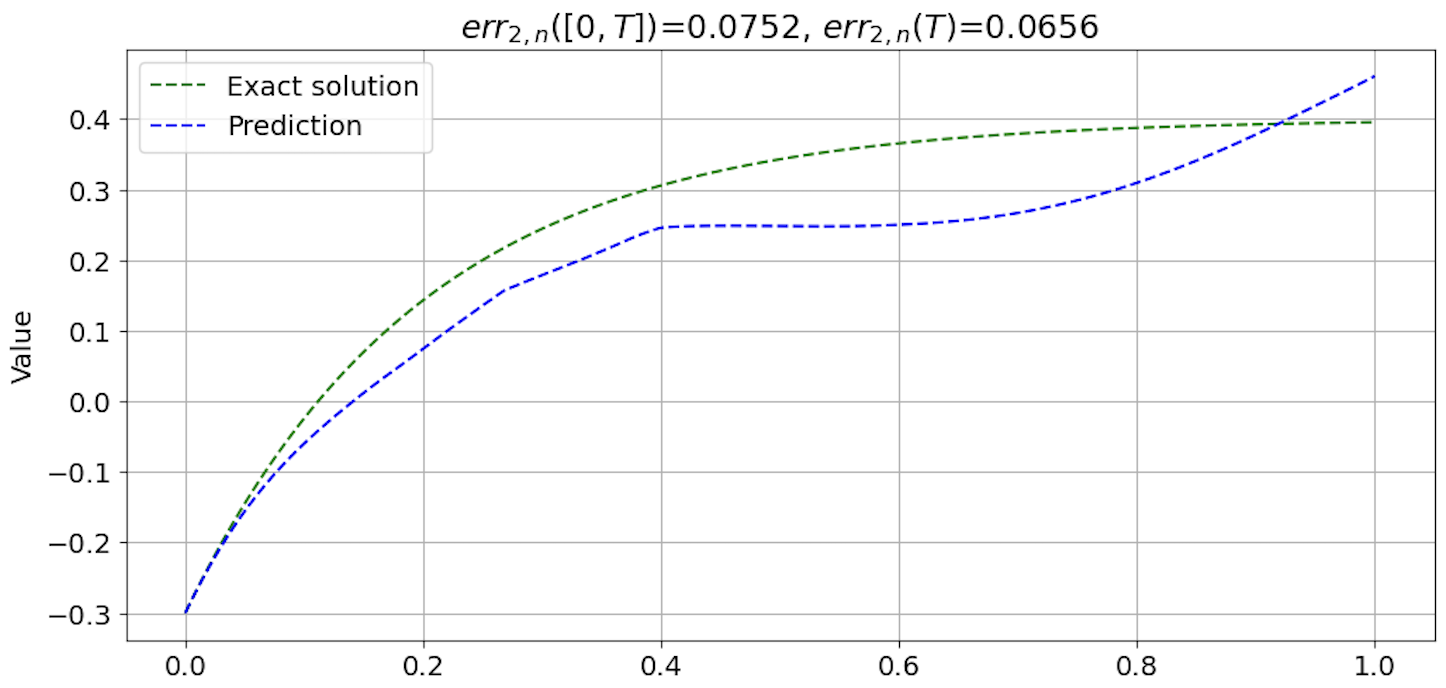}
        \caption{Network performance in solving ODE after training according to new approach.}
        \label{fig:ode_smart}
    \end{subfigure}
    
    \caption{Overall figure caption}
    \label{fig:bothplots}
\end{figure}

% \begin{figure} 
        
%     \end{figure}

We observed improved performance; however, there remains room for further optimization. Notably, the network struggles to accurately approximate the solution toward the end of the interval. To address this issue, we introduced a weighting factor by multiplying the loss function by $\tau$, thereby increasing the penalty as $\tau$ grows. The effects of this modification are illustrated in Figure \ref{fig:ode_tau}.

\begin{figure} 
        \centering
        \includegraphics[width=0.8\linewidth]{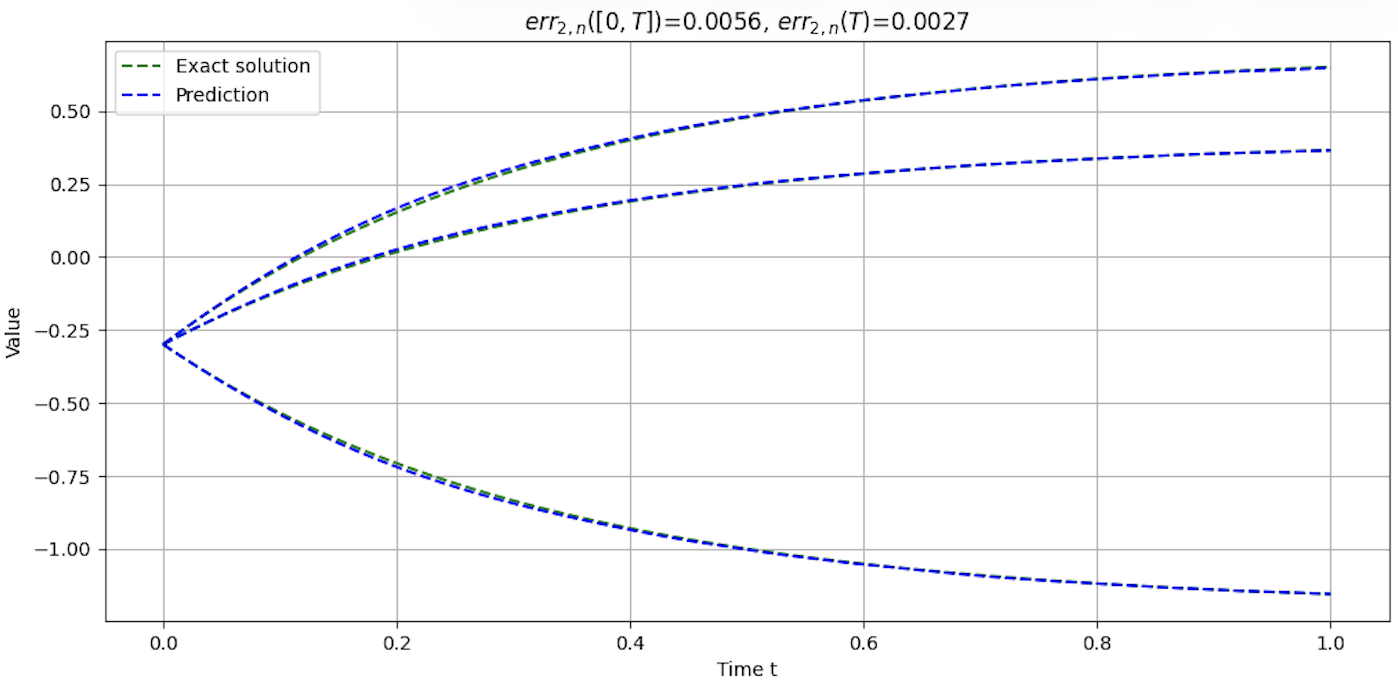}
        \caption{Network performance in solving ODE after training with weighting.}
        \label{fig:ode_tau}
\end{figure}

\subsection{Results}
This subsection focuses on visualizing the model’s performance. We examine the trajectories generated based on the model output and contrast them with the corresponding exact trajectories. This comparison offers insight into how well the learned dynamics replicate the system's true behavior.

In Figure \ref{fig:model_performance}, we present the performance of the model trained to solve equation in example 1. The green trajectory comes from the exact solution, and the blue is calculated by the neural network. The graphs with trajectories are entitled with values of $err_{2,512}([0,T])$ and $err_{2,512}(T)$, but in this case, the errors are calculated for $M=1$.

\begin{figure}[h]
    \centering
    \includegraphics[width=0.8\linewidth, height=10cm]{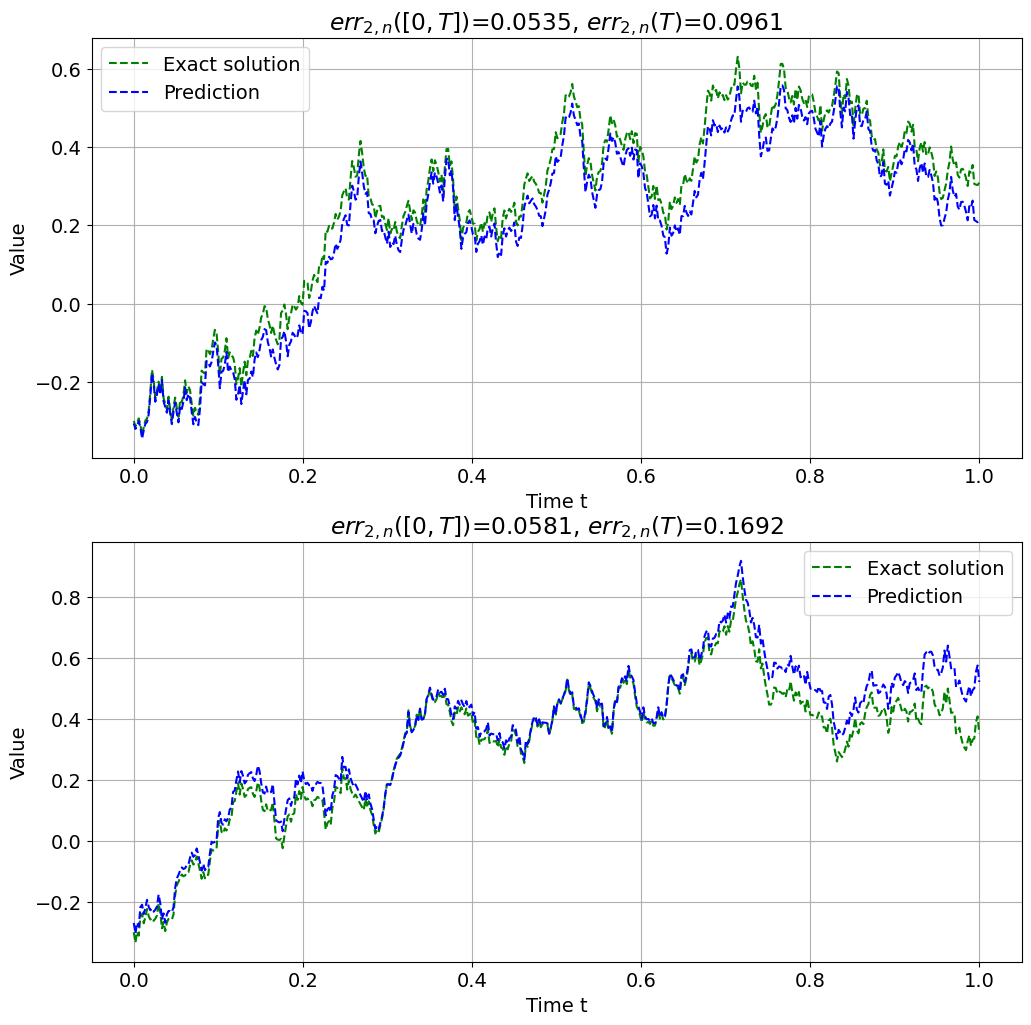}
    \caption{Model performance for example 1.}
    \label{fig:model_performance}
\end{figure}

% In Figure \ref{fig:model_performance_combined} we present performance for the same model as in Figure \ref{fig:model_performance}, but model is trained for 100 and 1000 epochs.

% \begin{figure}[h]
%     \centering
    
%     \begin{subfigure}{0.45\textwidth}
%         \centering
%         \includegraphics[width=\linewidth]{plots/performance_100_ex1.pdf}
%         % \caption{First plot}
%         % \label{fig:plot1}
%     \end{subfigure}
%     \hfill
%     \begin{subfigure}{0.45\textwidth}
%         \centering
%         \includegraphics[width=\linewidth]{plots/performance_1000_ex1.pdf}
%         % \caption{Second plot}
%         % \label{fig:plot2}
%     \end{subfigure}
    
%     \caption{Model performance with $n=4096$ for 2 trajectories after training for 100 and 1000 epochs for example 1.}
%     \label{fig:model_performance_combined}
% \end{figure}

The next chart (Fig.~\ref{fig:model_performance_ex2}) presents the performance of the model trained to solve example 2. The results indicate that the model’s performance is highly similar to that obtained for example 1, despite the presence of a nonlinear drift.

\begin{figure}[h]
    \centering
    \includegraphics[width=0.8\linewidth, height=10cm]{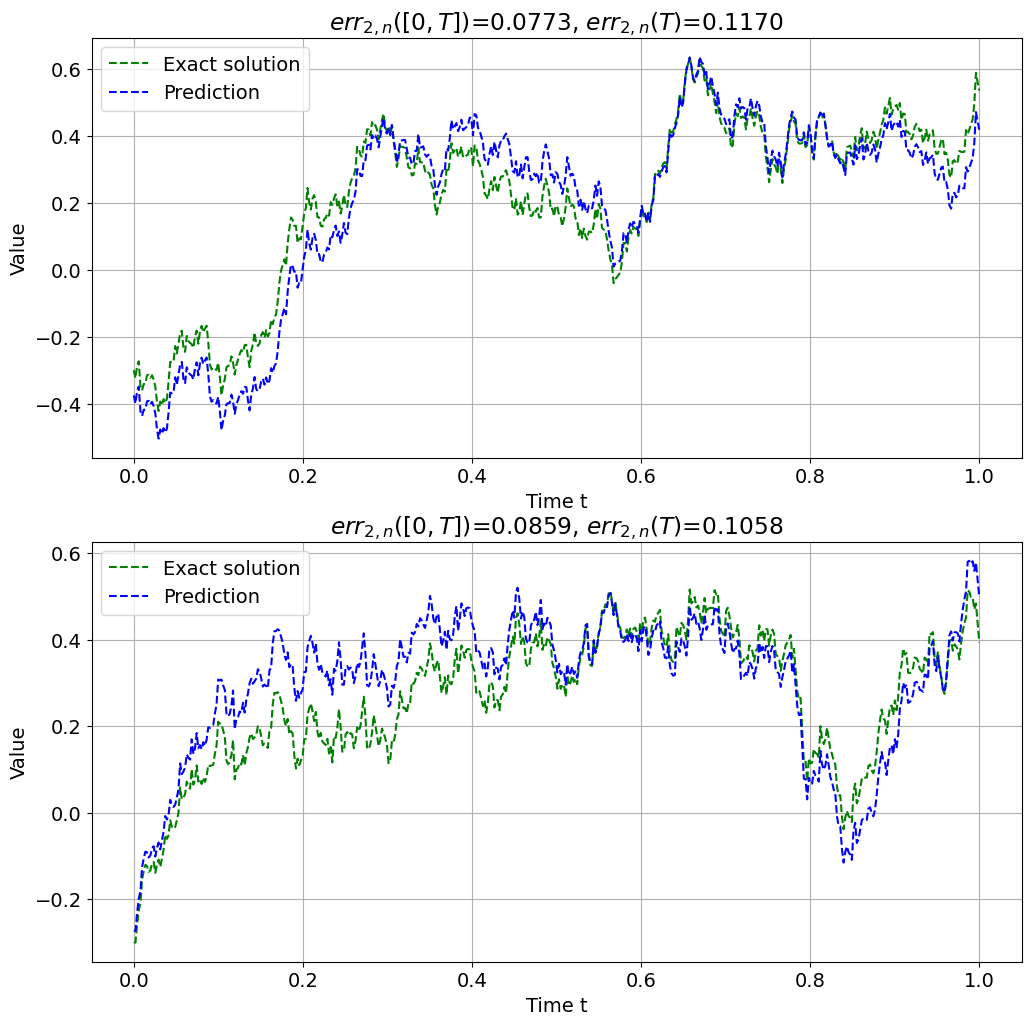}
    \caption{Model performance for example 2.}
    \label{fig:model_performance_ex2}
\end{figure}

Additionally, we examine the behavior of the neural network for Example 1, using as the process 
$L$ both the Poisson process and a composite process consisting of the sum of a Poisson process and a Wiener process.
In Figure \ref{fig:poisson_wiener} we present sample trajectories for such selection of process $L$. The charts on the left side shows trajectories driven by a Poisson process and charts on the right side presents trajectories driven by a sum of Poisson and Wiener processes.

\begin{figure}[h!]
    \centering
    % ---- First plot ----
    \begin{minipage}[b]{0.48\textwidth}
        \centering
        \includegraphics[width=\textwidth]{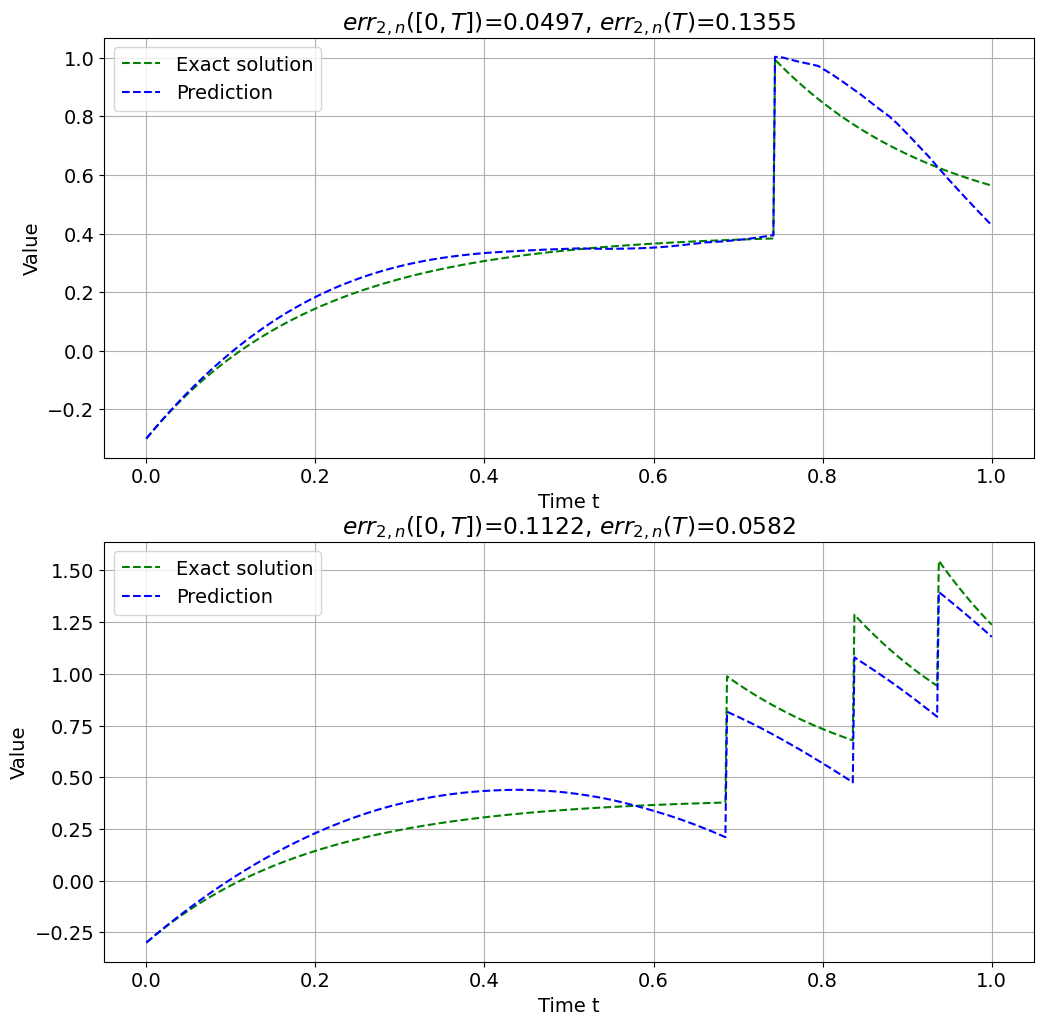}
    \end{minipage}
    \hfill
    % ---- Second plot ----
    \begin{minipage}[b]{0.48\textwidth}
        \centering
        \includegraphics[width=\textwidth]{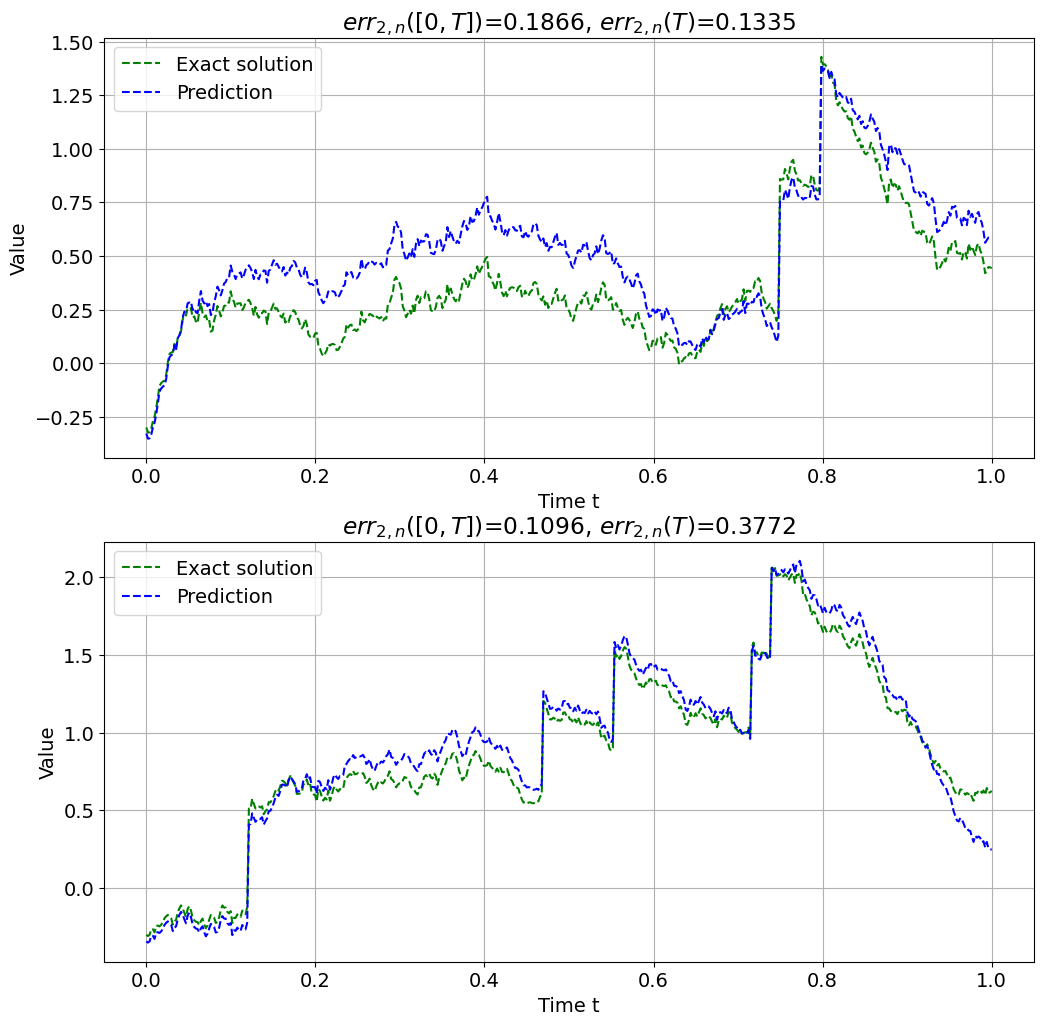}
    \end{minipage}

    \caption{Networks sample trajectories driven by Poisson and sum of Poisson and Wiener process for example 1.}
    \label{fig:poisson_wiener}
\end{figure}

We examine the behavior of the neural network for Example 2, noting that this case corresponds to a nonlinear setting. In this experiment, the driving process 
$L$ is again considered in two variants: a Poisson process and a composite process given by the sum of a Poisson process and a Wiener process.
In Figure \ref{fig:sin_poisson_wiener} we present sample trajectories obtained for this configuration. The charts on the left-hand side show trajectories generated when the system is driven by a Poisson process, while the charts on the right-hand side illustrate trajectories corresponding to the composite process formed by the sum of Poisson and Wiener processes.

\begin{figure}[h!]
    \centering
    % ---- First plot ----
    \begin{minipage}[b]{0.48\textwidth}
        \centering
        \includegraphics[width=\textwidth]{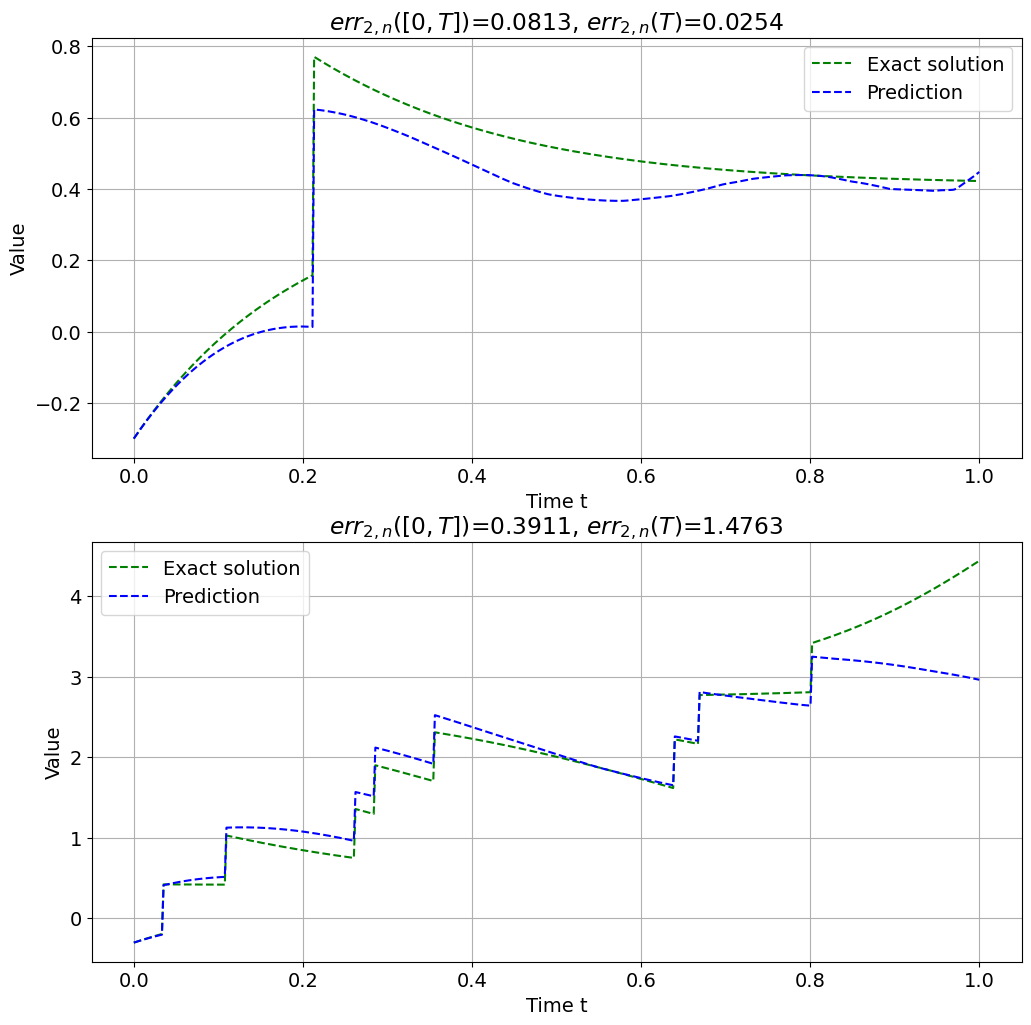}
    \end{minipage}
    \hfill
    % ---- Second plot ----
    \begin{minipage}[b]{0.48\textwidth}
        \centering
        \includegraphics[width=\textwidth]{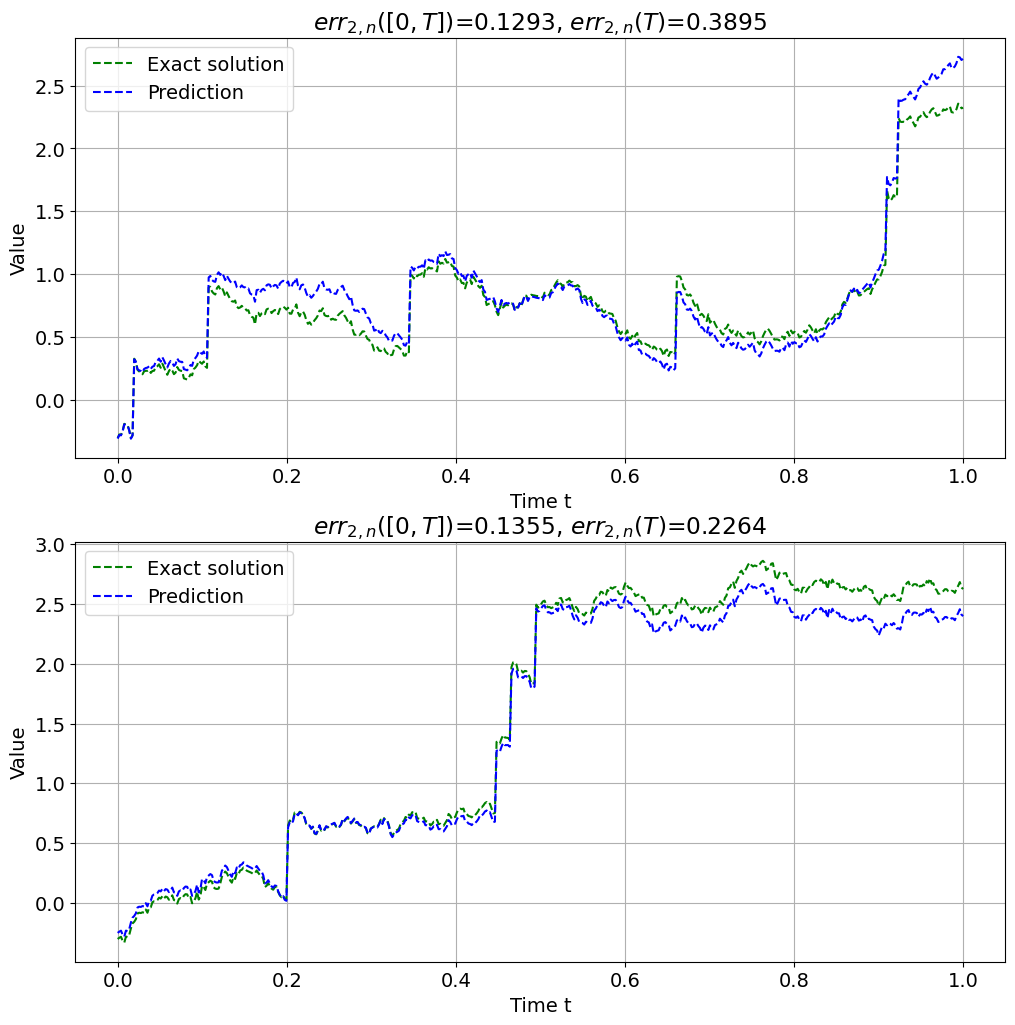}
    \end{minipage}

    \caption{Networks sample trajectories driven by Poisson and sum of Poisson and Wiener process for example 2.}
    \label{fig:sin_poisson_wiener}
\end{figure}

We examine the multidimensional Example 3 in two cases. In the first case, $L$ is a four-dimensional Wiener process, while in the second case, it is a four-dimensional process given by the sum of a Wiener and a Poisson process. The results are presented in Figure~\ref{fig:3dim}.

\begin{figure}[h!]
    \centering
    % ---- First plot ----
    \begin{minipage}[b]{0.48\textwidth}
        \centering
        \includegraphics[width=\textwidth]{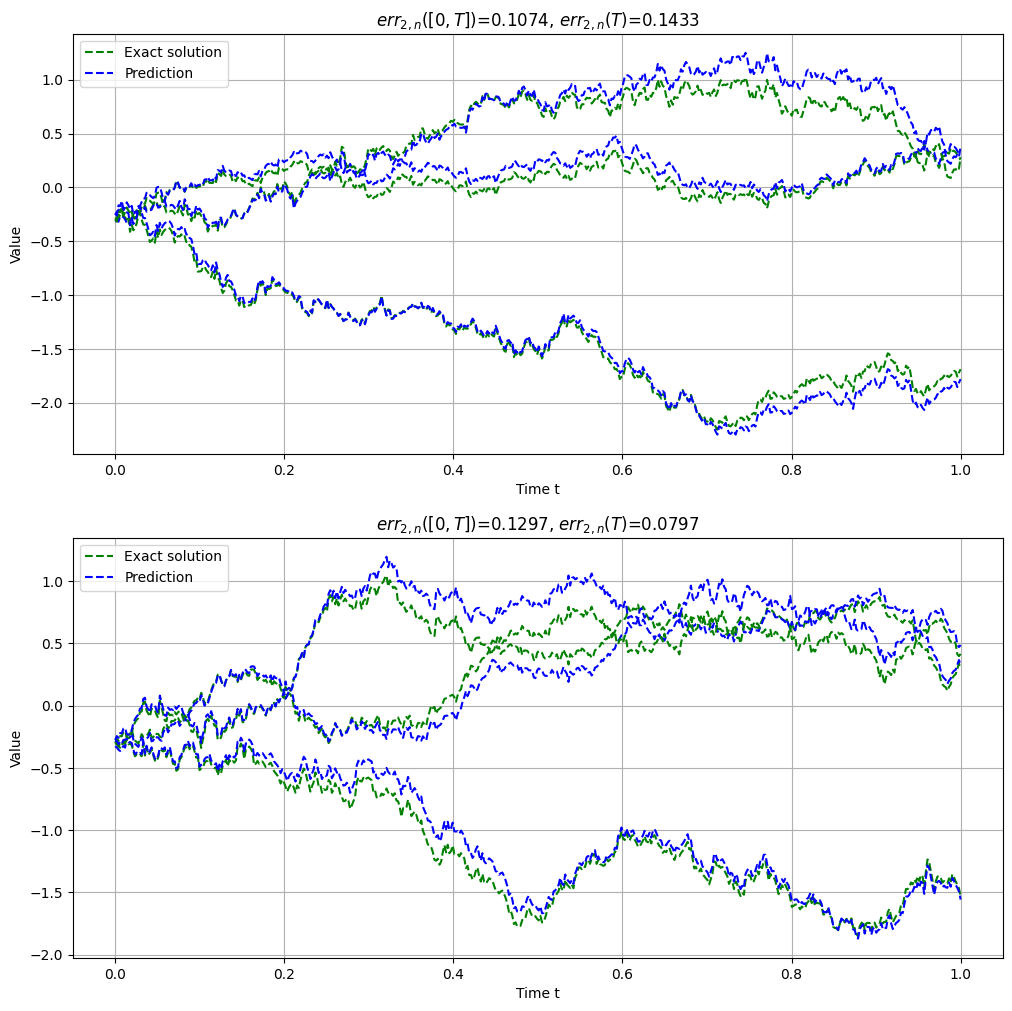}
    \end{minipage}
    \hfill
    % ---- Second plot ----
    \begin{minipage}[b]{0.48\textwidth}
        \centering
        \includegraphics[width=\textwidth]{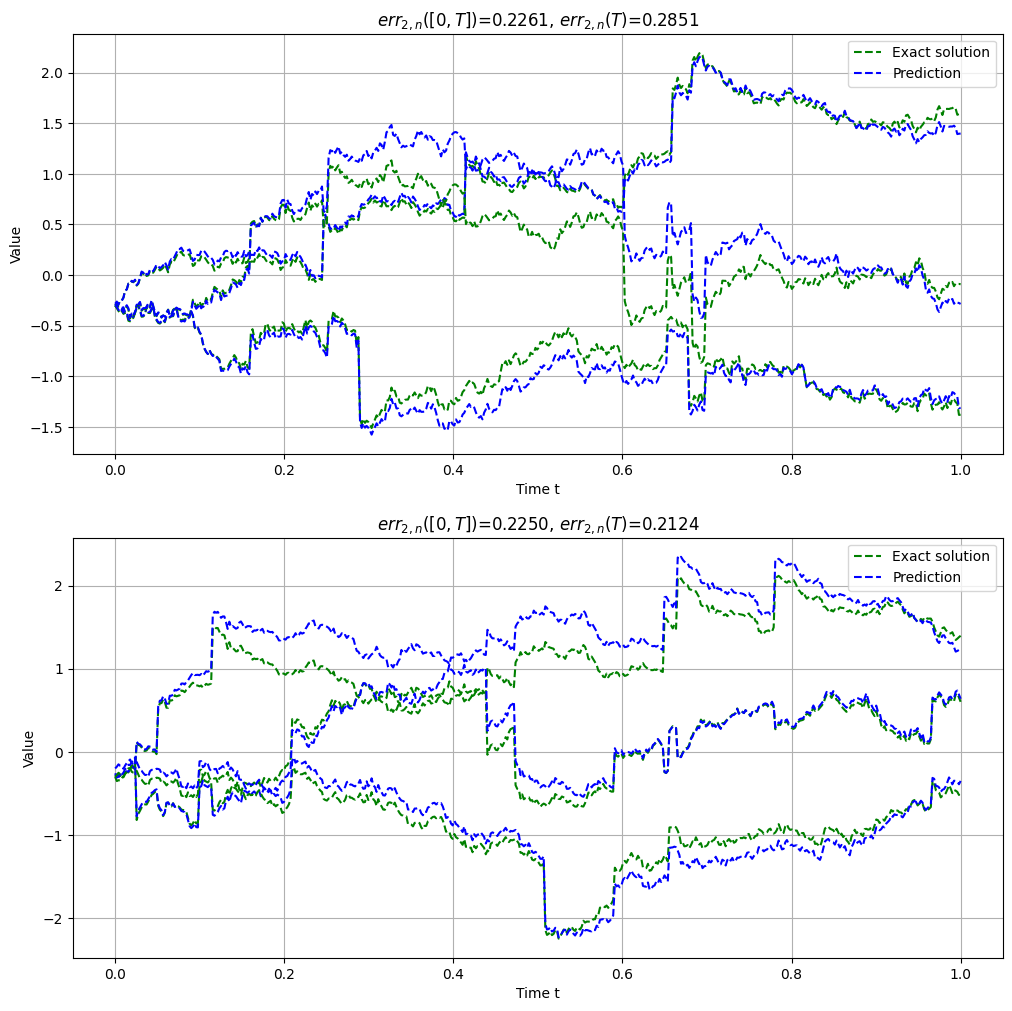}
    \end{minipage}

    \caption{Networks sample trajectories driven by Wiener and sum of Poisson and Wiener process for example 3.}
    \label{fig:3dim}
\end{figure}

% The plot in Fig.~\ref{fig:model_performance_ex3} illustrates the performance of the model trained to solve example~3. In this case, the drift function is discontinuous, which leads to increased difficulty for the model and results in poorer approximation accuracy.

% In Figure \ref{fig:sin} we present sample trajectories for example 2. where Wiener is the underlying process.

% \begin{figure}[h!]
%     \centering
%     \includegraphics[width=1\linewidth]{plots/sin.png}
%     % \centering
%     % % ---- First plot ----
%     % \begin{minipage}[b]{0.48\textwidth}
%     %     \centering
%     %     \includegraphics[width=\textwidth]{plots/poisson.png}
%     % \end{minipage}
%     % \hfill
%     % % ---- Second plot ----
%     % \begin{minipage}[b]{0.48\textwidth}
%     %     \centering
%     %     \includegraphics[width=\textwidth]{plots/poisson_wiener.png}
%     % \end{minipage}

%     \caption{Networks sample trajectories driven by Wiener process for example 2.}
%     \label{fig:sin}
% \end{figure}

% \begin{figure}[h]
%     \centering
%     \includegraphics[width=1\linewidth]{plots/performance_5000_ex3.pdf}
%     \caption{Model performance with $N=4096$ for 2 trajectories after training for 5000 epochs for example 3.}
%     \label{fig:model_performance_ex3}
% \end{figure}

%%%%%%%%%%
\section{Appendix}
%%%%%%%%%%%%
\begin{lemma}
\label{equiv_def_c1}
    For all $(\mathcal{F}_t)_{t\geq 0}$-adapted and continuous $\mathbb{R}^d$-valued stochastic processes $(y(t))_{t\in [0,T]}$, that have $(\mathcal{F}_t)_{t\geq 0}$-adapted and c\`adl\`ag  derivative processes $(y'(t))_{t\in [0,T]}$, the following conditions are equivalent:
    \begin{itemize}
        \item [(i)] $\max\Bigl\{\mathbb{E}\|y(0)\|^2,\mathbb{E}\int\limits_0^T\|y'(t)\|^2dt\Bigr\}<+\infty$,
        \item [(ii)] $\|y\|_1<+\infty$.
    \end{itemize}
\end{lemma}
\begin{proof}
    The implication (ii)$\Rightarrow$ (i) is straightforward. We now show (i)$\Rightarrow$ (ii). By Theorem 11 in \cite{HT_1} we have almost surely for all $t\in [0,T]$ that
    \begin{equation}
        y(t)=y(0)+\int\limits_0^t y'(s)ds.
    \end{equation}
    Therefore
    \begin{equation}
        \mathbb{E}\Bigl(\sup\limits_{0\leq t\leq T}\|y(t)\|^2\Bigr)\leq 2\mathbb{E}\|y(0)\|^2+2T\mathbb{E}\int\limits_0^T\|y'(t)\|^2dt<+\infty,
    \end{equation}
    which implies that $\|y\|_1<+\infty$.
\end{proof}
We present here auxiliary properties of the processes $Y$.
\begin{proposition}
\label{prop_sol_Y}
\begin{itemize}
    \item [(i)]
    Under assumptions (A1), (A2), the following hold:
    \begin{equation}
    \label{est_Y_1}
        \mathbb{P}\Biggl(\sup\limits_{0\leq t\leq T}\|Y(t)\|\leq C_3\Bigl(1+\int\limits_0^T\|L(t)\|dt\Bigr)<+\infty\Biggr)=1,
    \end{equation}
        \begin{equation}
        \label{est_DY_1}\mathbb{P}\Biggl(\int\limits_0^T\|Y'(t)\|^2dt\leq C_4\Bigl(1+\int\limits_0^T\|L(t)\|^2dt\Bigr)<+\infty\Biggr)=1,
    \end{equation}
    where
    \begin{equation}
        C_3=e^{C_2T}(1+\|x_0\|)\max\{1,C_2,C_2T\},
    \end{equation}
    \begin{equation}
        C_4=3C_2^2\max\{1,T\}(1+2C_3^2 T),
    \end{equation}
    and $C_2$ is defined in \eqref{def_C2}.
    \item [(ii)] Under assumptions (A1), (A2), (A3) the following hold:  
    \begin{equation}
        \mathbb{P}\Bigl(\sup\limits_{0\leq t\leq T}\|Y(t)\|\leq \|x_0\|+D_0T\Bigr)=1,
    \end{equation}
    \begin{equation}
\mathbb{P}\Bigl(\int\limits_0^T\|Y'(t)\|^2dt\leq D_0^2T\Bigr)=1.
    \end{equation}
    \end{itemize}
\end{proposition}
\begin{proof}
    We have with probability one for all $t\in [0,T]$ that
    \begin{eqnarray}
        &\|Y(t)\|\leq \|x_0\|+\int\limits_0^t \|f(s,Y(s),L(s))\|ds\leq \|x_0\|+C_2\int\limits_0^t (1+\|Y(s)\|+\|L(s)\|)ds\notag\\
        &\leq \max\{1,C_2,C_2T\}\Bigl(1+\|x_0\|+\int\limits_0^T\|L(s)\|ds\Bigr)+C_2\int\limits_0^t\|Y(s)\|ds.
    \end{eqnarray}
    Since almost all trajectories of $(Y(t))_{t\in [0,T]}$ are continuous and $\mathbb{P}\Bigl(\int\limits_0^T\|L(s)\|^2ds<+\infty\Bigr)=1$, we get by the Gronwall lemma that for all $t\in [0,T]$
    \begin{equation}
        \|Y(t)\|\leq e^{C_2 t}\max\{1,C_2,C_2T\}\Bigl(1+\|x_0\|+\int\limits_0^T\|L(s)\|ds\Bigr)
    \end{equation}
    almost surely. This gives \eqref{est_Y_1}. The estimate \eqref{est_DY_1} follows from \eqref{est_Y_1}, Fact \ref{prop_f} and the fact that
    \begin{equation}
        \int\limits_0^T\|Y'(t)\|^2dt\leq 3C_2^2\Bigl(T+\int\limits_0^T\|Y(t)\|^2dt+\int\limits_0^T\|L(t)\|^2dt\Bigr).
    \end{equation}
    The proof of (ii) is straightforward and, therefore, omitted.
\end{proof}
%%%%%%%%
Directly from Proposition \ref{prop_sol_Y} we get the following corollary.
\begin{corollary}
\label{cor_Y_in_C1}
    \begin{itemize}
        \item [(i)] If $\mathbb{E}\int\limits_0^T\|L(t)\|^2dt<+\infty$ and assumptions (A1), (A2) hold then the unique solution $(Y(t))_{t\in [0,T]}$ of \eqref{main_equation23} belongs to $\mathcal{D}^1\left([0,T]\times\Omega;\mathbb{R}^d\right)$.
        \item [(ii)] If $\mathbb{E}\int\limits_0^T\|L(t)\|^2dt=+\infty$ and assumptions (A1), (A2), (A3) hold then the unique solution $(Y(t))_{t\in [0,T]}$ of \eqref{main_equation23} belongs to $\mathcal{D}^1\left([0,T]\times\Omega;\mathbb{R}^d\right)$.
    \end{itemize}
\end{corollary}
We believe that the lemma below is well-known. However, we were unable to find a direct statement or proof, so, for the reader's convenience, we provide its justification.
\begin{lemma}
\label{zero_lem_1}
Let $f:[0,T]\to[0,+\infty)$ be c\`adl\`ag and $
\int_{0}^{T} f(t)dt = 0$. Then $f(t)=0$ for all $t\in[0,T)$.
\end{lemma}
\begin{proof} 
Fix any $t_{0}\in[0,T)$. Suppose for contradiction that $f(t_{0})>0$.
Since $f$ is right-continuous at $t_{0}$, there exists $\delta>0$ such that for all $t\in[t_{0},t_{0}+\delta)\cap [0,T]$
\begin{displaymath}
|f(t)-f(t_{0})|<\frac{f(t_{0})}{2},
\end{displaymath}
and hence $f(t) >\frac{f(t_{0})}{2}>0$, for all $t\in[t_{0},t_{0}+\delta)\cap [0,T]$.
Then
\[
\int_{0}^{T} f(t)\,dt
\;\ge\;
\int\limits_{[t_0,t_0+\delta)\cap [0,T]} f(t)dt\geq
\min\{\delta,T-t_0\}\frac{f(t_{0})}{2}>0,
\]
contradicting the assumption that $\int_{0}^{T} f(t)\,dt = 0$.
Hence $f(t_{0})=0$. Since $t_{0}\in[0,T)$ was arbitrary, we conclude that $f(t)=0$ for all  $t\in[0,T)$, which ends the proof.
\end{proof}
%%%%%%%%%%%
%%%%%%%%%%
\section{Conclusion and future work}
In this paper, we have introduced the StPINNs, the extension of the classical PINNs. Our approach allows us to solve SDEs \eqref{main_equation}, driven by an additive L\'evy process, with artificial neural networks. The derivation performed within this paper shows that there is a rich mathematical structure behind StPINNs, and throughout mathematical investigations should be done in order to state convergence theorems, etc., rigorously. In our future work, we plan to address some of these problems and also consider the SDEs with multiplicative noise case.
%%%%%%%%%%%%%%%%%

\end{document}